\documentclass[onefignum,onetabnum]{siamart190516}

\newsiamremark{assumption}{Assumption}
\crefname{assumption}{Assumption}{Assumptions}

\newsiamremark{condition}{Condition}
\crefname{condition}{Condition}{Condition}

\usepackage{amsfonts}
\usepackage{graphicx}
\usepackage{epstopdf}
\usepackage{algorithmic}
\ifpdf
  \DeclareGraphicsExtensions{.eps,.pdf,.png,.jpg}
\else
  \DeclareGraphicsExtensions{.eps}
\fi

\newsiamremark{remark}{Remark}
\newsiamremark{hypothesis}{Hypothesis}
\crefname{hypothesis}{Hypothesis}{Hypotheses}
\newsiamthm{claim}{Claim}

\headers{Viscosity solutions of HJBI equations for time-delay systems}{A. Plaksin}

\title{Viscosity solutions of Hamilton-Jacobi-Bellman-Isaacs equations for time-delay systems\thanks{Submitted to the editors DATE.
\funding{This work is supported by the Grant of the President of the Russian Federation (project no. MK-3566.2019.1).}}}

\author{Anton Plaksin\thanks{N.N. Krasovskii Institute of Mathematics and Mechanics
of the Ural Branch of the Russian Academy of Sciences, Yekaterinburg, Russia; Ural Federal University, Yekaterinburg, Russia
  (\email{a.r.plaksin@gmail.com}).}
}

\usepackage{amsopn}

\ifpdf
\hypersetup{
  pdftitle={Viscosity solutions of HJBI equations for time-delay systems},
  pdfauthor={A. Plaksin}
}
\fi

\begin{document}

\maketitle

% REQUIRED
\begin{abstract}
The paper deals with a zero-sum differential game for a dynamical system which motion is described by a nonlinear delay differential equation under an initial condition defined by a piecewise continuous function. The corresponding Cauchy problem for Hamilton-Jacobi-Bellman-Isaacs equation with coinvariant derivatives is derived and the definition of a viscosity solution of this problem is considered. It is proved that the differential game has the value that is the unique viscosity solution. Moreover, based on notions of sub- and superdifferentials corresponding to coinvariant derivatives, the infinitesimal description of the viscosity solution is obtained. The example of applying these results is given.
\end{abstract}

% REQUIRED

%\vspace{-0.1cm}

% REQUIRED

%\vspace{-0.1cm}

\section{Introduction}

In finite-horizon optimal control and differential game problems for dynamical systems described by ordinary differential equations, investigations of value functions are closely related to Cauchy problems for the corresponding Hamilton-Jacobi (HJ) equations (which are called Hamilton-Jacobi-Bellman (HJB) equations for optimal control problems and Hamilton-Jacobi-Bellman-Isaacs (HJBI) equations for differential games) with partial derivatives. In the case when a value function is differentiable, it is the classical (differentiable) solution of the Cauchy problem. In the general case, this function is the minimax and viscosity solutions of the Cauchy problem. It is known (see, e.g., \cite{Bardi_Capuzzo-Dolcetta_1997, Crandall_Evans_Lions_1984, Crandall_Lions_1983, Evans_1998, Krasovskii_Subbotin_1988, Misztela_2019, Subbotin_1995, Subbotin_1980}) that both of these solutions are unique and coincide with each other.

Turning to the discussion of similar optimal control and differential game problems for dynamical systems described by delay differential equations (time-delay systems), first of all, note that the natural analog of value functions in such problems are (see, e.g., \cite{Banks_1968, Banks_Manitius_1974, Krasovskii_1962, Osipov_1971}) value functionals on a space of motion histories. Thus, these problems become related to infinite-dimensional HJ equations. Wherein, such equations are considered with Frechet derivatives (see, e.g., \cite{Barbu_Barron_Jensen_1988, Barron_1990, Cannarsa_Da_Prato_1990, Cannarsa_Frankowska_1992, Clarke_Ledyaev_1994, Crandall_Lions_1985, Soner_1988}) as well as with various derivatives along right extensions such as coinvariant derivatives (see, e.g., \cite{Bayraktar_Keller_2018, Kim_1999, Lukoyanov_2000, Lukoyanov_2001, Lukoyanov_2006, Lukoyanov_2010b, Lukoyanov_Gomoyunov_Plaksin_2017, Pepe_Ito_2012, Plaksin_2019}), Clio derivatives (see \cite{Aubin_Haddad_2002}), and derivatives within Ito calculus framework (see, e.g., \cite{Dupire_2019, Pham_Zhang_2014, Saporito_2019}).

The theory of viscosity solutions of infinite-dimensional HJ equations with Frechet derivatives was investigated by many authors (see, e.g. \cite{Barbu_Barron_Jensen_1988, Barron_1990, Cannarsa_Da_Prato_1990, Cannarsa_Frankowska_1992, Clarke_Ledyaev_1994, Crandall_Lions_1985, Soner_1988}). Its application to optimal control problems for time-delay systems was considered in \cite{Barron_1990, Soner_1988}. In \cite{Soner_1988}, using the modified definition of the viscosity solution, its existence, uniqueness and coincidence with the value functional were proved. However, note that conditions in this paper allow to consider time-delay systems with distributed delay, but exclude the consideration of another important for applications case of systems with a discrete delay. Optimal control problems for systems with a discrete delay were investigated in \cite{Barron_1990}. This paper established that the value functional is the viscosity solution, but the uniqueness question of this solution was not investigated.

One can also mention papers \cite{Clarke_Wolenski_1996, Wolenski_1994} in which optimization problems for quite general delay differential inclusions (which cover the case of discrete delay) were considered and various necessary optimality conditions were given.

In \cite{Kim_1999}, for a description of the value functional in optimal control problems for time-delay systems, the notion of coinvariant derivatives was suggested. The theory of minimax and viscosity solutions of Cauchy problems for HJ equations with coinvariant derivatives and its application to differential games for time-delay systems were developed in \cite{Lukoyanov_2000, Lukoyanov_2001, Lukoyanov_2006, Lukoyanov_2010b}. In these papers, the class of time-delay systems under consideration is quite general and includes systems with both distributed and discrete delays. In \cite{Lukoyanov_2000, Lukoyanov_2001, Lukoyanov_2006}, it was shown that the value functional is the unique minimax solution. In \cite{Lukoyanov_2010b}, using, similarly to \cite{Soner_1988}, modified definition of the viscosity solution based on a sequence of compact sets, it was proved that the value functional is the unique viscosity solution. However, due to such specifics of the definition of the viscosity solution, it is not reduced to the classical definition of a viscosity solution in the particular case without delay.

In recent paper \cite{Plaksin_2019} dedicated to optimal control problems for time-delay systems with a discrete delay and the corresponding Cauchy problems for HJB equations with coinvariant derivatives, the viscosity solution of the Cauchy problem was defined by means of inequalities for sub- and superdifferentials corresponding to coinvariant derivatives. Such definition of the viscosity solution seems quite close to the more classical definition, since, in the classical theory of HJ equations with partial derivatives, viscosity solutions can be described by similar inequalities for usual sub- and superdifferentials. This paper \cite{Plaksin_2019} established that the value functional is the unique viscosity solution. Note that this result was managed to get, in particular, due to the choice of the space of piecewise continuous functions as the space of motion histories.

The present paper continues and, in a sense, expands investigations in \cite{Plaksin_2019}. We consider a differential game for a quite general time-delay system (which include discrete and distributed delays) on the space of piecewise continuous motion histories. We associate this game with Cauchy problem for the HJBI equation with coinvariant derivatives. We give the definition of a viscosity solution of this problem in classical sense. The main result of this paper is that the differential game has the value functional that is the unique viscosity solution. Moreover, similarly to \cite{Plaksin_2019}, based on notions of sub- and superdifferentials corresponding to coinvariant derivatives, we obtain the infinitesimal description of the viscosity solution. Also, we give formulas for calculating such sub- and superdifferentials and consider the example showing how we can apply all these results to a given functional in order to prove that it is the value functional of the differential game.

%%%%%%%%%%%%%%%%%%%%%%%%%%%%%%%%%%%%%%%%%%%%%%%%%%%%%%%%%%%%%%%%%%%%%%%%%%%%%%%%%%%%%%%%
%%%%%%%%%%%%%%%%%%%%%%%%%%%%%%%%%%%%%%%%%%%%%%%%%%%%%%%%%%%%%%%%%%%%%%%%%%%%%%%%%%%%%%%%
%%%%%%%%%%%%%%%%%%%%%%%%%%%%%%%%%%%%%%%%%%%%%%%%%%%%%%%%%%%%%%%%%%%%%%%%%%%%%%%%%%%%%%%%
%%%%%%%%%%%%%%%%%%%%%%%%%%%%%%%%%%%%%%%%%%%%%%%%%%%%%%%%%%%%%%%%%%%%%%%%%%%%%%%%%%%%%%%%
%%%%%%%%%%%%%%%%%%%%%%%%%%%%%%%%%%%%%%%%%%%%%%%%%%%%%%%%%%%%%%%%%%%%%%%%%%%%%%%%%%%%%%%%
%%%%%%%%%%%%%%%%%%%%%%%%%%%%%%%%%%%%%%%%%%%%%%%%%%%%%%%%%%%%%%%%%%%%%%%%%%%%%%%%%%%%%%%%

\section{Main results}\label{section:main_results}

Let $\mathbb R^n$ be the $n$-dimensional Euclidian space with the inner product $\langle \cdot, \cdot \rangle$ and the norm $\|\cdot\|$. A function $x(\cdot) \colon [a,b) \mapsto \mathbb R^n$ (or $x(\cdot) \colon [a,b] \mapsto \mathbb R^n$) is called piecewise continuous if there exist numbers $a = \xi_1 < \xi_2 < \ldots < \xi_k = b$ such that, for each $i \in \overline{1,k-1}$, the function $x(\cdot)$ is continuous on the interval $[\xi_i,\xi_{i+1})$ and there exists a finite limit of $x(\xi)$ as $\xi$ approaches $\xi_{i+1}$ from the left. Denote by $\mathrm{PC}([a,b), \mathbb R^n)$ (or $\mathrm{PC}([a,b], \mathbb R^n)$) the linear space of piecewise continuous functions $x(\cdot) \colon [a,b) \mapsto \mathbb R^n$ (or $x(\cdot) \colon [a,b] \mapsto \mathbb R^n$).

Let $t_0 < \vartheta$ and $h > 0$. Denote
\begin{equation}\label{PC_G}
\mathrm{PC} = \mathrm{PC}([-h,0),\mathbb R^n),\quad \mathbb G = [t_0,\vartheta] \times \mathbb R^n \times \mathrm{PC}.
\end{equation}
%Below in the paper, $\mathrm{PC}$ is space of motion histories and $\mathbb G$ is space of admissible positions.
Define the following norms on the space $\mathrm{PC}$:
%\vspace{-0.1cm}
\begin{displaymath}
\|w(\cdot)\|_1 = \int\limits_{-h}^0 \|w(\xi)\|\mathrm{d}\xi,\quad \|w(\cdot)\|_\infty = \sup\limits_{\xi\in [-h,0)} \|w(\xi)\|,\quad w(\cdot) \in \mathrm{PC}.
\end{displaymath}

For each $(\tau,z,w(\cdot)) \in \mathbb G$, we consider a two-player zero-sum differential game for the dynamical system described by the delay differential equation
\begin{equation}\label{dynamical_system}
\begin{array}{c}
\dot{x}(t) = f(t,x(t),x_t(\cdot),u(t),v(t)),\quad t \in [\tau,\vartheta],\\[0.2cm]
x(t) \in \mathbb R^n,\quad u(t) \in \mathbb U \subset \mathbb R^l,\quad v(t) \in \mathbb V \subset \mathbb R^m,
\end{array}
\end{equation}
with the initial condition
\begin{equation}\label{initial_condition}
x(\tau) = z,\quad x(t) = w(t - \tau),\quad t \in [\tau-h,\tau),
\end{equation}
and the quality index
\begin{equation}\label{quality_index}
\gamma = \sigma(x(\vartheta),x_\vartheta(\cdot)) + \int\limits_\tau^\vartheta f^0(t,x(t),x_t(\cdot),u(t),v(t)) \mathrm{d} t,
\end{equation}

Here $t$ is the time variable; $x(t)$ is the value of the state vector at the time $t$; $\dot{x}(t) = \mathrm{d} x(t) / \mathrm{d} t$; the symbol $x_t(\cdot)$ denotes the function on the interval $[-h,0)$ defined by $x_t(\xi) = x(t + \xi)$, $\xi \in [-h,0)$; $u(t)$ and $v(t)$ are control actions of the first and second players, respectively; $\mathbb U$ and $\mathbb V$ are compact sets.

In this differential game, the first player aims to minimize $\gamma$, while the second player aims to maximize it.

Denote
\begin{equation}\label{P}
P(\alpha) = \big\{(x,r(\cdot)) \in \mathbb R^n \times \mathrm{PC} \colon \|x\| \leq \alpha,\, \|r(\cdot)\|_\infty \leq \alpha\big\},\quad \alpha \geq 0.
\end{equation}

We assume that the following conditions hold:
\begin{condition}\label{cond:f_continuous}
For each $(x,r(\cdot)) \in \mathbb R^n \times \mathrm{PC}$, the mappings $[t_0,\vartheta] \times \mathbb R^l \times \mathbb R^m \ni (t,u,v) \mapsto f = f(t,x,r(\cdot),u,v) \in \mathbb R^n$ and $[t_0,\vartheta] \times \mathbb R^l \times \mathbb R^m \ni (t,u,v) \mapsto f^0 = f^0(t,x,r(\cdot),u,v) \in \mathbb R$ are continuous.
\end{condition}

\begin{condition}\label{cond:f_sublinear_growth}
There exists a constant $c_f > 0$ such that
\begin{displaymath}
\big\|f(t,x,r(\cdot),u,v)\big\| + \big|f^0(t,x,r(\cdot),u,v)\big|\leq c_f \big(1 + \|x\| + \|r(\cdot)\|_1 + \|r(-h)\|\big)
\end{displaymath}
for any $(t,x,r(\cdot)) \in \mathbb G$, $u \in \mathbb U$, and $v \in \mathbb V$.
\end{condition}

\begin{condition}\label{cond:f_lipshiz_continuous}
For every $\alpha > 0$, there exists a number $\lambda_f = \lambda_f(\alpha) > 0$ such that
\begin{displaymath}
\begin{array}{l}
\big\|f(t,x,r(\cdot),u,v) - f(t,x',r'(\cdot),u,v)\big\|
+ \big|f^0(t,x,r(\cdot),u,v) - f^0(t,x',r'(\cdot),u,v)\big|\\[0.2cm]
\hspace{3cm}
\leq \lambda_f \big(\|x - x'\| + \|r(\cdot) - r'(\cdot)\|_1 + \|r(-h) - r'(-h)\|\big)
\end{array}
\end{displaymath}
for any $t \in [t_0,\vartheta]$, $(x,r(\cdot)), (x',r'(\cdot)) \in P(\alpha)$, $u \in \mathbb U$, and $v \in \mathbb V$.
\end{condition}

\begin{condition}\label{cond:f_saddle_point}
The equality
\begin{displaymath}
\begin{array}{l}
\min\limits_{u \in \mathbb U} \max\limits_{v \in \mathbb V} \Big(\langle f(t,x,r(\cdot),u,v),s \rangle + f^0(t,x,r(\cdot),u,v)\Big)\\[0.2cm]
\hspace{2cm}
= \max\limits_{v \in \mathbb V} \min\limits_{u \in \mathbb U} \Big(\langle f(t,x,r(\cdot),u,v),s \rangle + f^0(t,x,r(\cdot),u,v)\Big)
\end{array}
\end{displaymath}
holds for any $(t,x,r(\cdot)) \in \mathbb G$ and $s \in \mathbb R^n$.
\end{condition}

\begin{condition}\label{cond:sigma_lipshiz_continuous}
The mapping $\mathbb R^n \times \mathrm{PC} \ni (x,r(\cdot)) \mapsto \sigma = \sigma(x,r(\cdot)) \in \mathbb R$ satisfies the following Lipschitz  continuity condition: for every $\alpha > 0$, there exists a number $\lambda_\sigma = \lambda_\sigma(\alpha) > 0$ such that
\begin{displaymath}
\big|\sigma(x,r(\cdot)) - \sigma(x',r'(\cdot))\big| \leq \lambda_\sigma \big(\|x - x'\| + \|r(\cdot) - r'(\cdot)\|_1\big)
\end{displaymath}
for any $(x,r(\cdot)),(x',r'(\cdot)) \in P(\alpha)$.
\end{condition}

\begin{remark}
\cref{cond:f_sublinear_growth,cond:f_lipshiz_continuous} allow to consider dynamical systems with both distributed and one discrete delays. For example, a particular case of system \cref{dynamical_system} is
\begin{displaymath}
\dot{x}(t) = f_*\Bigg(t,x(t), x(t-h),\int\limits_{t-h}^t g_*(t,\xi,x(\xi)) \mathrm{d} \xi, u(t), v(t)\Bigg),\quad t \in [t_0,\vartheta],
\end{displaymath}
%\begin{displaymath}
%\gamma = \sigma_*\Bigg(x(\vartheta),\int\limits_{\vartheta-h}^\vartheta g_2(\xi,x(\xi)) \mathrm{d} \xi\Bigg) + \int\limits_\tau^\vartheta f^0_*\Bigg(t,x(t),\int\limits_{t-h}^t g_3(t,\xi,x(\xi)) \mathrm{d} \xi, x(t-h), u(t), v(t)\Bigg) \mathrm{d} t,
%\end{displaymath}
under suitable conditions for the functions $f_*$ and $g_*$.
\end{remark}

\begin{remark}
In the case when system \cref{dynamical_system} does not have delays, \cref{cond:f_continuous,cond:f_sublinear_growth,cond:f_lipshiz_continuous,cond:f_saddle_point} take the form of classical conditions for the optimal control and differential games theories (see, e.g., \cite{Evans_1998, Krasovskii_Subbotin_1988, Subbotin_1995}).
\end{remark}

Define the set of Lipschitz continuous right extensions from $(\tau,z,w(\cdot))$ as follows:
\begin{displaymath}
\begin{array}{rl}
\Lambda(\tau,z,w(\cdot)) = \big\{x(\cdot) \in \mathrm{PC}([\tau-h,\vartheta],\mathbb R^n) \colon \!\!\!\!\!\! &  x(\tau) = z,\, x(t) = w(t - \tau),\, t \in [\tau-h,\tau),\\[0.2cm]
& x(\cdot) \text{ is Lipschitz continuous on } [\tau,\vartheta] \big\}.
\end{array}
\end{displaymath}
By admissible control realizations of the first and second players, we mean measurable functions $u(\cdot) \colon [\tau,\vartheta] \mapsto \mathbb U$ and $v(\cdot) \colon [\tau,\vartheta] \mapsto \mathbb V$, respectively. Denote by $\mathcal{U}_\tau$ and $\mathcal{V}_\tau$ the sets of admissible control realizations of the first and second players. Under \cref{cond:f_continuous,cond:f_sublinear_growth,cond:f_lipshiz_continuous}, following, for example, the scheme from \cite[Section 7]{Filippov_1988} (see also \cite[Section 4.2]{Kim_1999}), one can show that each pair of realizations $u(\cdot) \in \mathcal{U}_\tau$ and $v(\cdot) \in \mathcal{V}_\tau$ uniquely generate the motion $x(\cdot) = x(\cdot\,|\,\tau,z,w(\cdot),u(\cdot),v(\cdot))$ of system \cref{dynamical_system,initial_condition} that is the function from $\Lambda(\tau,z,w(\cdot))$, satisfying equation \cref{dynamical_system} almost everywhere.

We consider differential game \cref{dynamical_system,initial_condition,quality_index} in classes of non-anticipative strategies of players (see, e.g. \cite[Chapter VIII, Section 1]{Bardi_Capuzzo-Dolcetta_1997}) or quasi-strategies in another terminology (see, e.g. \cite[Chapter III, Section 14.2]{Subbotin_1995}).

By a non-anticipative strategy of the first player, we mean a mapping $Q^u_\tau \colon \mathcal{V}_\tau \mapsto \mathcal{U}_\tau$ such that, for each $v(\cdot),v'(\cdot) \in \mathcal{V}_\tau$ and $t \in [\tau,\vartheta]$, if the equality $v(\xi) = v'(\xi)$ is valid for a.e. $\xi \in [\tau,t]$, then the equality $Q^u_\tau[v(\cdot)](\xi) = Q^u_\tau[v'(\cdot)](\xi)$ holds for a.e. $\xi \in [\tau,t]$.

A non-anticipative strategy of the first player $Q^u_\tau$ and a control realization of the second player $v(\cdot) \in \mathcal{V}_\tau$ define the control realization of the first player $u(\cdot) = Q^u_\tau[v(\cdot)](\cdot)$, the motion $x(\cdot) = x(\cdot\,|\,\tau,z,w(\cdot),u(\cdot),v(\cdot))$ of system \cref{dynamical_system,initial_condition} and the value $\gamma = \gamma(\tau,z,w(\cdot),Q^u_\tau,v(\cdot))$ of quality index \cref{quality_index}. The lower value of differential game \cref{dynamical_system,initial_condition,quality_index} is defined by
\begin{equation}\label{rho_u}
\rho^u(\tau,z,w(\cdot)) = \inf\limits_{Q^u_\tau} \sup\limits_{v(\cdot) \in \mathcal{V}_\tau} \gamma(\tau,z,w(\cdot),Q^u_\tau,v(\cdot)).
\end{equation}
The functional $\mathbb G \ni (\tau,z,w(\cdot)) \mapsto \rho^u = \rho^u(\tau,z,w(\cdot)) \in \mathbb R$ is the lower value functional of differential game \cref{dynamical_system,initial_condition,quality_index}.

Similarly, a non-anticipative strategy of the second player is a mapping $Q^v_\tau \colon \mathcal{U}_\tau \mapsto \mathcal{V}_\tau$ such that, for each $u(\cdot),u'(\cdot) \in \mathcal{U}_\tau$ and $t \in [\tau,\vartheta]$, if the equality $u(\xi) = u'(\xi)$ is valid for a.e. $\xi \in [\tau,t]$, then the equality  $Q^v_\tau[u(\cdot)](\xi) = Q^v_\tau[u'(\cdot)](\xi)$ holds for a.e. $\xi \in [\tau,t]$. Such a non-anticipative strategy together with a control realization of the first player $u(\cdot) \in \mathcal{U}_\tau$ define the control realization of the second player $v(\cdot) = Q^v_\tau[u(\cdot)](\cdot)$, the motion $x(\cdot) = x(\cdot\,|\,\tau,z,w(\cdot),u(\cdot),v(\cdot))$ of system \cref{dynamical_system,initial_condition} and the value $\gamma = \gamma(\tau,z,w(\cdot),u(\cdot),Q^v_\tau)$. The upper value of differential game \cref{dynamical_system,initial_condition,quality_index} is
\begin{equation}\label{rho_v}
\rho^v(\tau,z,w(\cdot)) = \sup\limits_{Q^v_\tau} \inf\limits_{u(\cdot) \in \mathcal{U}_\tau} \gamma(\tau,z,w(\cdot),u(\cdot),Q^v_\tau).
\end{equation}
The functional $\mathbb G \ni (\tau,z,w(\cdot)) \mapsto \rho^v = \rho^v(\tau,z,w(\cdot)) \in \mathbb R$ is the upper value functional of differential game \cref{dynamical_system,initial_condition,quality_index}.

If the lower value functional $\rho^u$ and the upper value functional $\rho^v$ satisfy the equality $\rho^u(\tau,z,w(\cdot)) = \rho^v(\tau,z,w(\cdot))$ for any $(\tau,z,w(\cdot)) \in \mathbb G$ then we say that differential game \cref{dynamical_system,initial_condition,quality_index} has the value functional
\begin{displaymath}
\rho = \rho(\tau,z,w(\cdot)) = \rho^u(\tau,z,w(\cdot)) = \rho^v(\tau,z,w(\cdot)),\quad (\tau,z,w(\cdot)) \in \mathbb G.
\end{displaymath}

In order to consider Hamilton-Jacobi (HJ) equation, which corresponds to differential game \cref{dynamical_system,initial_condition,quality_index}, we use the following definition of differentiability of functionals. Following \cite{Kim_1999, Lukoyanov_2000}, a functional $\varphi \colon \mathbb G \mapsto \mathbb R$ is called coinvariantly (ci-) differentiable at a point $(\tau,z,w(\cdot)) \in \mathbb G$, $\tau < \vartheta$ if there exist a number $\partial^{ci}_{\tau,w}\varphi(\tau,z,w(\cdot)) \in \mathbb R$ and a vector $\nabla_z\varphi(\tau,z,w(\cdot)) \in \mathbb R^n$ such that, for every $t \in [\tau,\vartheta]$, $x \in \mathbb R^n$, and $y(\cdot) \in \Lambda(\tau,z,w(\cdot))$, the relation below holds
\begin{displaymath}
\begin{array}{c}
\varphi(t,x,y_t(\cdot)) - \varphi(\tau,z,w(\cdot))  =  (t - \tau) \partial^{ci}_{\tau,w}\varphi(\tau,z,w(\cdot)) \\[0.2cm]
+ \langle x - z, \nabla_z \varphi(\tau,z,w(\cdot)) \rangle + o(|t - \tau| + \|x - z\|),
\end{array}
\end{displaymath}
where the value $o(\cdot)$ depends on the triplet $\{\tau,z,y(\cdot)\}$, and $o(\delta)/\delta \to 0$ as $\delta \to +0$. Then $\partial^{ci}_{\tau,w}\varphi(\tau,z,w(\cdot))$ is called the ci-derivative of $\varphi$ with respect to $\{\tau,w(\cdot)\}$ and $\nabla_z \varphi(\tau,z,w(\cdot))$ is the gradient of $\varphi$ with respect to $z$.

Define the Hamiltonian of differential game \cref{dynamical_system,initial_condition,quality_index} by
\begin{equation}\label{Hamiltonian}
\begin{array}{c}
H(\tau,z,w(\cdot),s) = \min\limits_{u \in \mathbb U} \max\limits_{v \in \mathbb V} \big(\langle f(\tau,z,w(\cdot),u,v),s\rangle + f^0(\tau,z,w(\cdot),u,v)\big),\\[0.2cm]
(\tau,z,w(\cdot)) \in \mathbb G,\quad s \in \mathbb R^n.
\end{array}
\end{equation}
Consider the following Cauchy problem for the HJ equation
\begin{equation}\label{Hamilton-Jacobi_equation}
\partial^{ci}_{\tau,w} \varphi(\tau,z,w(\cdot)) + H(\tau,z,w(\cdot),\nabla_z \varphi(\tau,z,w(\cdot))) = 0,\ \
(\tau,z,w(\cdot)) \in \mathbb G,\ \ \tau < \vartheta,
\end{equation}
and the terminal condition
\begin{equation}\label{terminal_condition}
\varphi(\vartheta,z,w(\cdot)) =\sigma(z,w(\cdot)),\quad (\vartheta,z,w(\cdot)) \in \mathbb G.
\end{equation}

Similarly to \cite{Plaksin_2019}, we search a solution of this problem in the following class of functionals. Denote by $\Phi$ the set of functionals $\varphi = \varphi(\tau,z,w(\cdot)) \in \mathbb R$, $(\tau,z,w(\cdot)) \in \mathbb G$, which are continuous with respect to $\tau$ and satisfy the following Lipschitz continuity condition: for every $\alpha > 0$, there exists a number $\lambda_\varphi = \lambda_\varphi(\alpha) > 0$ such that
\begin{equation}\label{phi_lipshiz_continuous}
|\varphi(\tau,z,w(\cdot)) - \varphi(\tau,z',w'(\cdot))| \leq \lambda_\varphi\big(\|z - z'\| + \|w(\cdot) - w'(\cdot)\|_1\big)
\end{equation}
for any $\tau \in [t_0,\vartheta]$ and $(z,w(\cdot)), (z',w'(\cdot)) \in P(\alpha)$.

\begin{remark}
The choice of this class is motivated, in particular, by the inclusions $\rho^u, \rho^v \in \Phi$, which are proved in \cref{lem:rho_in_Phi} below.
\end{remark}

%^Define the one element set containing extension from $(\tau,z,w(\cdot))$ by constant
Define the singleton consisting of the right extension from $(\tau,z,w(\cdot)) \in \mathbb G$ by constant:
\begin{displaymath}
\Lambda_0(\tau,z,w(\cdot)) = \big\{x(\cdot) \in \Lambda(\tau,z,w(\cdot))\colon x(t) = z,\, t \in [\tau,\vartheta]\big\} \subset \Lambda(\tau,z,w(\cdot)).
\end{displaymath}

\begin{remark}\label{rem:ci-differentiable}
One can show that a functional $\varphi \in \Phi$ is ci-differentiable at a point $(\tau,z,w(\cdot)) \in \mathbb G$, $\tau < \vartheta$ if and only if there exist $\partial^{ci}_{\tau,w}\varphi(\tau,z,w(\cdot)) \in \mathbb R$ and $\nabla_z\varphi(\tau,z,w(\cdot)) \in \mathbb R^n$ such that
\begin{displaymath}
\begin{array}{c}
\varphi(t,x,\kappa_t(\cdot)) - \varphi(\tau,z,w(\cdot))  =  (t - \tau) \partial^{ci}_{\tau,w}\varphi(\tau,z,w(\cdot)) \\[0.2cm]
+ \langle x - z, \nabla_z \varphi(\tau,z,w(\cdot)) \rangle + o(|t - \tau| + \|x - z\|),
\end{array}
\quad t \in [\tau,\vartheta],\quad x \in \mathbb R^n,
\end{displaymath}
where $\kappa(\cdot) \in \Lambda_0(\tau,z,w(\cdot))$, $o(\cdot)$ depends on $\{\tau,z,w(\cdot)\}$, and $o(\delta)/\delta \to 0$ as $\delta \to +0$. Thus, for functionals $\varphi \in \Phi$, the definition of ci-differentiability becomes easier.
\end{remark}

Now, let us give two propositions which establish the relation between Cauchy problem \cref{Hamilton-Jacobi_equation,terminal_condition} and the value functional $\rho$ of differential game \cref{dynamical_system,initial_condition,quality_index} in the case when $\rho$ is ci-differentiable.

\begin{proposition}
Let differential game \cref{dynamical_system,initial_condition,quality_index} have the value functional $\rho$. If $\rho$ is ci-differentiable at a point $(\tau,z,w(\cdot)) \in \mathbb G$, $\tau < \vartheta$, then it satisfies HJ equation \cref{Hamilton-Jacobi_equation} at this point.
\end{proposition}
\begin{proposition}
Let a functional $\varphi \in \Phi$ be ci-differentiable at every point $(\tau,z,w(\cdot)) \in \mathbb G$, $\tau < \vartheta$, satisfy HJ equation \cref{Hamilton-Jacobi_equation} at these points and satisfy terminal condition \cref{terminal_condition}. Then $\varphi$ is the value functional of differential game \cref{dynamical_system,initial_condition,quality_index}.
\end{proposition}
These propositions follow from \cref{prop:smooth_solution_is_viscosity_solution,prop:viscosity_solution_is_smooth_solution,theorem} below.

For the cases when $\rho$ is not ci-differentiable, we will consider generalized (viscosity) solutions of problem \cref{Hamilton-Jacobi_equation,terminal_condition}. But let us first make an auxiliary remark.

\begin{remark}
Let $\varphi \in \Phi$ be ci-differentiable at a point $(\tau,z,w(\cdot)) \in \mathbb G$, $\tau < \vartheta$ and satisfy HJ equation \cref{Hamilton-Jacobi_equation} at this point. Define the function
\begin{equation}\label{tilde_phi}
\tilde{\varphi}(t,x) = \varphi(t,x,\kappa_t(\cdot)),\quad (t,x) \in \mathbb [\tau,\vartheta] \times \mathbb R^n,\quad \kappa(\cdot) \in \Lambda_0(\tau,z,w(\cdot)).
\end{equation}
Then one can show that at the point $(\tau,z)$ the function $\tilde{\varphi}$ has a right partial derivative $\partial^+ \tilde{\varphi}(\tau,z) / \partial \tau$ and a gradient $\nabla_z \tilde{\varphi}(\tau,z)$, and satisfies the following HJ equation:
\begin{equation}\label{usual_HJ}
\partial^+ \tilde{\varphi}(\tau,z) / \partial \tau + H(\tau,z,w(\cdot),\nabla_z \tilde{\varphi}(\tau,z)) = 0.
\end{equation}
\end{remark}
Thus, we might say that HJ equation with a ci-derivative \cref{Hamilton-Jacobi_equation} is locally HJ equation with partial derivatives \cref{usual_HJ}.
Then, based on the classical definition of viscosity solutions \cite{Crandall_Lions_1983} (see also \cite{Bardi_Capuzzo-Dolcetta_1997, Evans_1998}), it leads us in a natural way to the following definition of a viscosity solution of problem \cref{Hamilton-Jacobi_equation,terminal_condition}.

\begin{definition}\label{def:viscosity_solution}
A functional $\varphi \in \Phi$ is called a viscosity solution of problem \cref{Hamilton-Jacobi_equation,terminal_condition}, if it satisfies terminal condition \cref{terminal_condition} and
the following properties:
\begin{subequations}
\begin{gather}
\left\{
\begin{array}{ll}
\text{for every}\ (\tau,z,w(\cdot)) \in \mathbb G,\ \tau < \vartheta,\ \psi \in \mathrm{C}^1(\mathbb R \times \mathbb R^n,\mathbb R), \text{ and } \delta > 0,\\[0.2cm]
\text{if}\ \varphi(\tau,z,w(\cdot)) - \psi(\tau,z) \leq \varphi(t,x,\kappa_t(\cdot)) - \psi(t,x),\ (t,x) \in O_\delta^+(\tau,z),\\[0.2cm]
\text{then}\ \partial \psi(\tau,z) / \partial \tau + H(\tau,z,w(\cdot),\nabla_z \psi(\tau,z)) \leq 0,
\end{array} \right.\label{sub_viscosity_solution}\\[0.4cm]
\left\{
\begin{array}{ll}
\text{for every}\ (\tau,z,w(\cdot)) \in \mathbb G,\ \tau < \vartheta,\ \psi \in \mathrm{C}^1(\mathbb R \times \mathbb R^n,\mathbb R), \text{ and } \delta > 0,\\[0.2cm]
\text{if}\ \varphi(\tau,z,w(\cdot)) - \psi(\tau,z) \geq \varphi(t,x,\kappa_t(\cdot)) - \psi(t,x),\ (t,x) \in O_\delta^+(\tau,z),\\[0.2cm]
\text{then}\ \partial \psi(\tau,z) / \partial \tau + H(\tau,z,w(\cdot),\nabla_z \psi(\tau,z)) \geq 0.
\end{array}\label{super_viscosity_solution} \right.
\end{gather}
\end{subequations}
Here $\kappa(\cdot) \in \Lambda_0(\tau,z,w(\cdot))$,  $O^+_\delta(\tau,z) = \{(t,x) \in [\tau,\tau+\delta] \times \mathbb R^n\colon \|x - z\| \leq \delta\}$, and $\mathrm{C}^1(\mathbb R \times \mathbb R^n, \mathbb R)$ is the space of continuously differentiable functions from $\mathbb R \times \mathbb R^n$ to $\mathbb R$.
\end{definition}

Note that this definition consistent with a ci-dif\-fe\-ren\-tiable solution of problem \cref{Hamilton-Jacobi_equation,terminal_condition}. The following two propositions establish it.

\begin{proposition}\label{prop:viscosity_solution_is_smooth_solution}
Let a functional $\varphi \in \Phi$ be a viscosity solution of problem \cref{Hamilton-Jacobi_equation,terminal_condition}. If $\varphi$ is ci-differentiable at a point $(\tau,z,w(\cdot)) \in \mathbb G$, $\tau < \vartheta$, then $\varphi$ satisfies HJ equation \cref{Hamilton-Jacobi_equation} at this point.
\end{proposition}

\begin{proposition}\label{prop:smooth_solution_is_viscosity_solution}
Let a functional $\varphi \in \Phi$ be ci-differentiable at every point $(\tau,z,w(\cdot)) \in \mathbb G$, $\tau < \vartheta$, satisfy HJ equation \cref{Hamilton-Jacobi_equation} at these points and satisfy terminal condition \cref{terminal_condition}. Then $\varphi$ is a viscosity solution of problem \cref{Hamilton-Jacobi_equation,terminal_condition}.
\end{proposition}
These propositions follow from \cref{theorem,prop:subdifferential_inequalities,prop:subdifferentials_and_ci-differentials} below.

Let us give the main result of the paper.

\begin{theorem}\label{theorem}
Differential game \cref{dynamical_system,initial_condition,quality_index} has the value functional $\rho$ that is the unique viscosity solution of problem \cref{Hamilton-Jacobi_equation,terminal_condition}.
\end{theorem}
The theorem is based on two results proved in the paper: 1) the functionals $\rho^u$ and $\rho^v$ are viscosity solutions of problem \cref{Hamilton-Jacobi_equation,terminal_condition} (see \cref{lem:rho_is_viscosity_solution}); 2) problem \cref{Hamilton-Jacobi_equation,terminal_condition} has a unique viscosity solution (see \cref{uniqueness_of_viscosity solution}).

Next, similarly to the classical theory of viscosity solutions of HJ equations (see, e.g., \cite{Bardi_Capuzzo-Dolcetta_1997, Crandall_Evans_Lions_1984}), we introduce  notions of sub- and superdifferentials of functionals and, based on them, we give the infinitesimal description of the viscosity solution of problem \cref{Hamilton-Jacobi_equation,terminal_condition}.

The subdifferential of a functional $\varphi \colon \mathbb G \mapsto \mathbb R$ at a point $(\tau,z,w(\cdot)) \in \mathbb G$, $\tau < \vartheta$ is a set, denoted by $D^-(\tau,z,w(\cdot))$, of pairs $(p_0,p) \in \mathbb R \times \mathbb R^n$ such that
\begin{equation}\label{subdifferential}
\lim\limits_{\delta \to 0} \inf\limits_{(t,x) \in O^+_\delta(\tau,z)}
\frac{\varphi(t,x,\kappa_t(\cdot)) - \varphi(\tau,z,w(\cdot)) - (t - \tau) p_0 - \langle x - z, p \rangle}{|t - \tau| + \|x - z\|} \geq 0.
\end{equation}
The superdifferential of a functional $\varphi \colon \mathbb G \mapsto \mathbb R$ at a point $(\tau,z,w(\cdot)) \in \mathbb G$, $\tau < \vartheta$ is a set, denoted by $D^+(\tau,z,w(\cdot))$, of pairs $(q_0,q) \in \mathbb R \times \mathbb R^n$ such that
\begin{equation}\label{superdifferential}
\lim\limits_{\delta \to 0} \sup\limits_{(t,x) \in O^+_\delta(\tau,z)}
\frac{\varphi(t,x,\kappa_t(\cdot)) - \varphi(\tau,z,w(\cdot)) - (t - \tau) q_0 - \langle x - z, q \rangle}{|t - \tau| + \|x - z\|} \leq 0.
\end{equation}

\begin{theorem}\label{prop:subdifferential_inequalities}
A functional $\varphi \in \Phi$ is a viscosity solution of problem \cref{Hamilton-Jacobi_equation,terminal_condition} if and only if it satisfies terminal condition \cref{terminal_condition} and the inequalities
\begin{subequations}
\begin{gather}
p_0 + H(\tau,z,w(\cdot),p) \leq 0,\quad (p_0,p) \in D^-\varphi(\tau,z,w(\cdot)),\label{subdifferential_viscosity_solution}\\[0.1cm]
q_0 + H(\tau,z,w(\cdot),q) \geq 0,\quad (q_0,q) \in D^+\varphi(\tau,z,w(\cdot)),\label{superdifferential_viscosity_solution}
\end{gather}
\end{subequations}
for any $(\tau,z,w(\cdot)) \in \mathbb G$, $\tau < \vartheta$.
\end{theorem}
The theorem follows from \cref{lem:sub_viscosity_solution_and_subdifferential_viscosity_solution,lem:super_viscosity_solution_and_superdifferential_viscosity_solution}, which establish
the equivalence of property \cref{sub_viscosity_solution} and inequality \cref{subdifferential_viscosity_solution}, and property \cref{super_viscosity_solution} and inequality \cref{superdifferential_viscosity_solution}.

%where the equivalence of property \cref{sub_viscosity_solution} and inequality \cref{subdifferential_viscosity_solution}, and property \cref{super_viscosity_solution} and inequality \cref{superdifferential_viscosity_solution} is established.

At the end of the section, let us consider the question of calculation of sub- and superdifferentials. Below we give two ways for it. The first way can be used at points of ci-differentiability of a functional. The second way, based on the notion of directional derivatives, can be used in a general case.

\begin{proposition}\label{prop:subdifferentials_and_ci-differentials}
Let $\varphi \in \Phi$ and $(\tau,z,w(\cdot)) \in \mathbb G$, $\tau < \vartheta$. If $\varphi$ is ci-differentiable at $(\tau,z,w(\cdot))$ then the relations below hold:
\begin{displaymath}
\begin{array}{rcl}
D^-\varphi(t,z,w(\cdot)) \!\!\!\!&=&\!\!\!\! \big\{(p_0,p) \in \mathbb R \times \mathbb R^n \colon p_0 \leq \partial^{ci}_{\tau,w}\varphi(t,z,w(\cdot)),\, p = \nabla_z\varphi(t,z,w(\cdot))\big\},\\[0.2cm]
D^+\varphi(t,z,w(\cdot)) \!\!\!\!&=&\!\!\!\! \big\{(q_0,q) \in \mathbb R \times \mathbb R^n \colon q_0 \geq \partial^{ci}_{\tau,w}\varphi(t,z,w(\cdot)),\, q = \nabla_z\varphi(t,z,w(\cdot))\big\}.
\end{array}
\end{displaymath}
%Thus, such definition of sub- and superdifferentials correspond to the definition of ci-differentiability.
\end{proposition}
The proposition directly follows from \cref{rem:ci-differentiable} and definitions of sub- and superdifferentials (see \cref{subdifferential,superdifferential}).

Denote $\mathbb R^+ = [0,+\infty)$. Similarly to \cite{Lukoyanov_2006}, lower and upper right directional derivatives of a functional $\varphi \colon \mathbb G \mapsto \mathbb R$ along a direction $(l_0,l) \in \mathbb R^+ \times \mathbb R^n$ at a point $(\tau,z,w(\cdot)) \in \mathbb G$, $\tau < \vartheta$ are defined by
\begin{subequations}\label{direct_deriv}
\begin{align}
\partial^-_{(l_0,l)} \varphi(\tau,z,w(\cdot)) &
= \liminf\limits_{
\delta \to + 0\atop
(g_0,g) \to (l_0,l)} \displaystyle\frac{\varphi(g_0 \delta + \tau, g \delta + z,\kappa_{g_0 \delta + \tau}(\cdot)) - \varphi(\tau,z,w(\cdot))}{\delta},\label{direct_deriv:lower} \\[0.0cm]
\partial^+_{(l_0,l)} \varphi(\tau,z,w(\cdot)) &
= \limsup\limits_{\delta \to + 0 \atop (g_0,g) \to (l_0,l)} \displaystyle\frac{\varphi(g_0 \delta + \tau, g \delta + z,\kappa_{g_0 \delta + \tau}(\cdot)) - \varphi(\tau,z,w(\cdot))}{\delta},\label{direct_deriv:upper}
\end{align}
\end{subequations}
where $\kappa(\cdot) \in \Lambda_0(\tau,z,w(\cdot))$ and $(g_0, g) \in \mathbb R^+ \times \mathbb R^n$.

\begin{proposition}\label{prop:subdifferentials_and_directional_derivatives}
Let $\varphi \in \Phi$ and $(\tau,z,w(\cdot)) \in \mathbb G$, $\tau < \vartheta$. Then the relations below hold:
\begin{displaymath}
\begin{array}{l}
D^-\varphi(\tau,z,w(\cdot)) = \big\{(p_0,p) \in \mathbb R \times \mathbb R^n\colon \\[0.2cm]
\hspace{3.4cm}
l_0 p_0 + \langle l,p \rangle \leq \partial^-_{(l_0,l)} \varphi(\tau,z,w(\cdot)),\ (l_0,l) \in \mathbb R^+\! \times \mathbb R^n\big\},\\[0.2cm]
D^+\varphi(\tau,z,w(\cdot)) = \big\{(q_0,q) \in \mathbb R \times \mathbb R^n\colon \\[0.2cm]
\hspace{3.4cm}
l_0 q_0 + \langle l,q \rangle \geq \partial^+_{(l_0,l)} \varphi(\tau,z,w(\cdot)),\ (l_0,l) \in \mathbb R^+\! \times \mathbb R^n\big\}.
\end{array}
\end{displaymath}
\end{proposition}
The proposition can be proved similarly to \cite[Chapter III, Lemma 2.37]{Bardi_Capuzzo-Dolcetta_1997}.

In the next section, we consider example which demonstrates how we can apply results of this section to a given functional in order to prove that it is the value functional of the game. The idea of this example is taken from paper \cite{Lukoyanov_2006}.

%%%%%%%%%%%%%%%%%%%%%%%%%%%%%%%%%%%%%%%%%%%%%%%%%%%%%%%%%%%%%%%%%%%%%%%%%%%%%%%%%%%%%%%%
%%%%%%%%%%%%%%%%%%%%%%%%%%%%%%%%%%%%%%%%%%%%%%%%%%%%%%%%%%%%%%%%%%%%%%%%%%%%%%%%%%%%%%%%
%%%%%%%%%%%%%%%%%%%%%%%%%%%%%%%%%%%%%%%%%%%%%%%%%%%%%%%%%%%%%%%%%%%%%%%%%%%%%%%%%%%%%%%%
%%%%%%%%%%%%%%%%%%%%%%%%%%%%%%%%%%%%%%%%%%%%%%%%%%%%%%%%%%%%%%%%%%%%%%%%%%%%%%%%%%%%%%%%
%%%%%%%%%%%%%%%%%%%%%%%%%%%%%%%%%%%%%%%%%%%%%%%%%%%%%%%%%%%%%%%%%%%%%%%%%%%%%%%%%%%%%%%%
%%%%%%%%%%%%%%%%%%%%%%%%%%%%%%%%%%%%%%%%%%%%%%%%%%%%%%%%%%%%%%%%%%%%%%%%%%%%%%%%%%%%%%%%

\section{Example} Let $\mathrm{PC} = \mathrm{PC}([-1,0),\mathbb R^2)$, $\mathbb G = [0,3] \times \mathbb R^2 \times \mathrm{PC}$. For each $(\tau,z,w(\cdot)) \in \mathbb G$, we consider the differential game for the dynamical system described by the delay differential equations
\begin{equation}\label{ex:dynamical_system}
\left\{
\begin{array}{ll}
\dot{x}_1(t) = x_2(t-1) + u(t),\\
\dot{x}_2(t) = v(t),
\end{array} \right.
\quad t \in [0,3],
\end{equation}
\begin{displaymath}
x(t) = (x_1(t),x_2(t)) \in \mathbb R^2,\quad |u(t)| \leq 1,\quad |v(t)| \leq 2,
\end{displaymath}
with initial condition \cref{initial_condition}, and the quality index
\begin{equation}\label{ex:quality_index}
\gamma = |x_1(3)| - \int_\tau^3 \sqrt{1 - u^2(t)} \mathrm{d} t.
\end{equation}

Let us define the functionals
\begin{equation}\label{ex:phis_phi}
\begin{array}{rcl}
\varphi_0(\tau,z,w(\cdot)) &=& z_1 + \chi(\tau) z_2 + \displaystyle\int\limits_{-1}^{2 - \tau - \chi(\tau)} w_2(\xi) \mathrm{d} \xi,\\[0.4cm]
\varphi_*(\tau,z,w(\cdot),\theta) &=& \big|\varphi_0(\tau,z,w(\cdot))\big| \theta + \eta(\tau, \theta),\quad \theta \in [0,1],\\[0.4cm]
\varphi(\tau,z,w(\cdot)) &=& \max\limits_{\theta \in [0,1]} \varphi_*(\tau,z,w(\cdot),\theta),\quad (\tau,z,w(\cdot)) \in \mathbb G,
\end{array}
\end{equation}
where $z = (z_1,z_2)$, $w(\xi) = (w_1(\xi),w_2(\xi))$, $\xi \in [-1,0)$, $\chi(\tau) = \max\{2 - \tau,0\}$ and
$\eta(\tau, \theta) = \chi^2(\tau) \theta - (3 - \tau)\sqrt{1 + \theta^2}$. The value $\varphi(\tau,z,w(\cdot))$ is the programmed maximin \cite[Section 5.1]{Krasovskii_Subbotin_1988} of differential game \cref{ex:dynamical_system,initial_condition,ex:quality_index}, i.e.
\begin{displaymath}
\varphi(\tau,z,w(\cdot)) = \sup\limits_{v(\cdot) \in \mathcal{V}_\tau}\inf\limits_{u(\cdot) \in \mathcal{U}_\tau} \gamma,
\end{displaymath}
where $\gamma$ is defined by \cref{ex:quality_index} under the motion $x(\cdot) = x(\cdot\,|\,\tau,z,w(\cdot),u(\cdot),v(\cdot))$ of system \cref{ex:dynamical_system,initial_condition}. In the general case, the programmed maximin can be less than the upper value of the game. Below, on the basis of results from the \cref{section:main_results}, we show that $\varphi$ is the value functional of differential game \cref{ex:dynamical_system,initial_condition,ex:quality_index}.

Note  that, for every $(\tau,z,w(\cdot)), (\tau',z',w'(\cdot)) \in \mathbb G$, if $\theta' \in [0,1]$ satisfies the equality $\varphi(\tau',z',w'(\cdot)) = \varphi_*(\tau',z',w'(\cdot),\theta')$ then
\begin{equation}\label{ex:phi_phis}
\varphi(\tau',z',w'(\cdot)) - \varphi(\tau,z,w(\cdot)) \leq \varphi_*(\tau',z',w'(\cdot),\theta') - \varphi_*(\tau,z,w(\cdot),\theta').
\end{equation}
Using this estimate together with \cref{ex:phis_phi}, one can establish the inclusion $\varphi \in \Phi$.

In our case of quality index \cref{ex:quality_index}, terminal condition \cref{terminal_condition} becomes $\sigma(z,w(\cdot)) = |z_1|$, $(z,w(\cdot)) \in \mathbb R^2 \times \mathrm{PC}$. It is easy to check that the functional $\varphi$ satisfies such condition. Thus, according to \cref{theorem,prop:subdifferential_inequalities}, in order to establish that $\varphi$ is the value functional, it is sufficient to get inequalities \cref{subdifferential_viscosity_solution,superdifferential_viscosity_solution}, in which, due to \cref{Hamiltonian}, we have
\begin{equation}\label{ex:Hamiltonian}
H(\tau,z,w(\cdot),s) = w_2(-1) s_1 - \sqrt{1 + s_1^2} + 2 |s_2|,\quad (\tau,z,w(\cdot)) \in \mathbb G,\quad s \in \mathbb R^2.
\end{equation}

Denote by $\mathbb G_+$, $\mathbb G_-$ and $\mathbb G_0$ the sets of points $(\tau,z,w(\cdot)) \in \mathbb G$, $\tau < \vartheta$ such that $\varphi_0(\tau,z,w(\cdot)) > 0$, $\varphi_0(\tau,z,w(\cdot)) < 0$ and $\varphi_0(\tau,z,w(\cdot)) = 0$, respectively.

Let us consider the case when $(\tau,z,w(\cdot)) \in \mathbb G_+ \cup \mathbb G_-$. Define the function $\tilde{\varphi}$ by \cref{tilde_phi} and the function $\tilde{\varphi}_*$ by
\begin{displaymath}
\tilde{\varphi}_*(t,x,\theta) = \varphi_*(t,x,\kappa_t(\cdot),\theta),\quad
t \in \mathbb [\tau,\vartheta],\quad x \in \mathbb R^n,\quad
\theta \in [0,1],
\end{displaymath}
where $\kappa(\cdot) \in \Lambda_0(\tau,z,w(\cdot))$. Then, taking into account \cref{ex:phis_phi}, one can show that these functions are continuous and, for each $t \in [\tau,\vartheta]$ and $x \in \mathbb R^n$, there exists a unique value $\theta(t,x) \in [0,1]$ such that $\tilde{\varphi}(t,x) = \tilde{\varphi}_*(t,x,\theta(t,x))$. It means that the function $(t,x) \mapsto \theta(t,x)$ is continuous. Denote $\theta_\circ = \theta(\tau,z)$. Then, using \cref{ex:phis_phi,ex:phi_phis,rem:ci-differentiable}, one can establish that the functional $\varphi$ is ci-differentiable at $(\tau,z,w(\cdot))$ and
\begin{displaymath}
\begin{array}{rcl}
\partial^{ci}_{\tau,w}\varphi(\tau,z,w(\cdot)) \!\!\!\!&=&\!\!\!\! \mp w_2(-1) \theta_\circ - 2 \chi(\tau) \theta_\circ + \sqrt{1 + \theta_\circ^2},\\[0.2cm]
\nabla_z\varphi(\tau,z,w(\cdot)) \!\!\!\!&=&\!\!\!\! (\pm \theta_\circ,\pm \chi(\tau) \theta_\circ),
\end{array}
\quad (\tau,z,w(\cdot)) \in \mathbb G_{\pm}.
\end{displaymath}
Hence, taking into account \cref{prop:subdifferentials_and_ci-differentials,ex:Hamiltonian}, we obtain that inequalities \cref{subdifferential_viscosity_solution,superdifferential_viscosity_solution} holds.

Let us consider the case when $(\tau,z,w(\cdot)) \in \mathbb G_0$. Then $\varphi$ is not ci-differentiable at $(\tau,z,w(\cdot))$, however, similarly to the previous reasoning, we can find the directional derivatives of $\varphi$ along $(l_0,l=(l_1,l_2)) \in \mathbb R^+\! \times \mathbb R^2$ (see \cref{direct_deriv:lower,direct_deriv:upper}):
\begin{equation}\label{ex:direct_deriv_phi}
\begin{array}{c}
\partial^-_{(l_0,l)} \varphi(\tau,z,w(\cdot)) = \partial^+_{(l_0,l)} \varphi(\tau,z,w(\cdot))\\[0.2cm]
\hspace{1.0cm}= |- w(-1) l_0 + l_1 + \chi(\tau) l_2| \theta_\circ + (- 2 \chi(\tau) \theta_\circ + \sqrt{1 + \theta_\circ^2}) l_0,
\end{array}
\end{equation}
Let $(p_0,p) \in D^- \varphi(\tau,z,w(\cdot))$, $p = (p_1,p_2) \in \mathbb R^2$. By \cref{prop:subdifferentials_and_directional_derivatives}, it means that $p_0 l_0 + p_1 l_1 + p_2 l_2 \leq \partial^-_{(l_0,l)} \varphi(\tau,z,w(\cdot))$ for any $l_0 \in \mathbb R$, $l = (l_1,l_2) \in \mathbb R^2$. Applying this inequality for $(l_0,l) \in \{(1,(w(-1),0)),(0,(\pm 1,0)),(0,(\pm\chi(\tau),\mp 1))\}$, and according to \cref{ex:direct_deriv_phi}, we get
\begin{equation}\label{ex:p0_p1_p2}
p_0 \leq - w_2(-1) p_1 - 2 \chi(\tau) \theta_\circ + \sqrt{1 + \theta_\circ^2},\quad |p_1| \leq \theta_\circ,\quad p_2 = \chi(\tau) p_1.
\end{equation}
Define the function $\nu(\theta) = - 2 \chi(\tau) \theta + \sqrt{1 + \theta^2}$, $\theta \in [0,1]$. Then from \cref{ex:Hamiltonian,ex:p0_p1_p2} we derive
\begin{equation}\label{ex:HJ_estimate}
p_0 + H(\tau,z,w(\cdot),p) \leq \nu(\theta_\circ) - \nu(|p_1|).
\end{equation}
Since $(\tau,z,w(\cdot)) \in \mathbb G_0$, according to \cref{ex:phis_phi}, we have $\theta_\circ = \mathrm{argmax}\{\eta(\tau,\theta)\,|\,\theta\in [0,1]\}$. Using this, one can show that the function $\nu$ is monotone decreasing on $[0,\theta_\circ]$. Then, from \cref{ex:HJ_estimate} and the second relation in \cref{ex:p0_p1_p2}, we obtain inequality \cref{subdifferential_viscosity_solution}. In the similar way, one can prove inequality \cref{superdifferential_viscosity_solution}.

Thus, $\varphi$ is the value functional of differential game \cref{ex:dynamical_system,ex:quality_index,initial_condition}.

%%%%%%%%%%%%%%%%%%%%%%%%%%%%%%%%%%%%%%%%%%%%%%%%%%%%%%%%%%%%%%%%%%%%%%%%%%%%%%%%%%%%%%%%
%%%%%%%%%%%%%%%%%%%%%%%%%%%%%%%%%%%%%%%%%%%%%%%%%%%%%%%%%%%%%%%%%%%%%%%%%%%%%%%%%%%%%%%%
%%%%%%%%%%%%%%%%%%%%%%%%%%%%%%%%%%%%%%%%%%%%%%%%%%%%%%%%%%%%%%%%%%%%%%%%%%%%%%%%%%%%%%%%
%%%%%%%%%%%%%%%%%%%%%%%%%%%%%%%%%%%%%%%%%%%%%%%%%%%%%%%%%%%%%%%%%%%%%%%%%%%%%%%%%%%%%%%%
%%%%%%%%%%%%%%%%%%%%%%%%%%%%%%%%%%%%%%%%%%%%%%%%%%%%%%%%%%%%%%%%%%%%%%%%%%%%%%%%%%%%%%%%
%%%%%%%%%%%%%%%%%%%%%%%%%%%%%%%%%%%%%%%%%%%%%%%%%%%%%%%%%%%%%%%%%%%%%%%%%%%%%%%%%%%%%%%%

\section{Auxiliary statements} Let us give the following auxiliary statements. They can be proved similarly to Propositions 3.1--3.3 in \cite{Plaksin_2019}.

\begin{lemma}\label{lem:alpha_x-lambda_x}
For every $\alpha > 0$, there exist $\alpha_X = \alpha_X(\alpha) > 0$ and $\lambda_X = \lambda_X(\alpha) > 0$ such that, for every $\tau \in [t_0,\vartheta]$, $(z,w(\cdot)) \in P(\alpha)$, $u(\cdot) \in \mathcal{U}_\tau$, and $v(\cdot) \in \mathcal{V}_\tau$, the motion $x(\cdot) = x(\cdot\,|\,\tau,z,w(\cdot),u(\cdot),v(\cdot))$ of system \cref{dynamical_system} satisfies the relations
\begin{displaymath}
(x(t),x_t(\cdot)) \in P(\alpha_X),\quad \|x(t) - x(t')\| \leq \lambda_X|t - t'|,\quad t,t' \in [\tau,\vartheta].
\end{displaymath}
\end{lemma}

\begin{lemma}\label{lem:motion_lipshiz_continuous}
For every $\alpha > 0$, there exists $\lambda_* = \lambda_*(\alpha) > 0$ such that, for every $\tau \in [t_0,\vartheta]$, $(z,w(\cdot)), (z',w'(\cdot)) \in P(\alpha)$, $u(\cdot) \in \mathcal{U}_\tau$, and $v(\cdot) \in \mathcal{V}_\tau$, the motions $x(\cdot) = x(\cdot\,|\,\tau,z,w(\cdot),u(\cdot),v(\cdot))$ and $x'(\cdot) = x(\cdot\,|\,\tau,z',w'(\cdot),u(\cdot),v(\cdot))$ of system \cref{dynamical_system} satisfy the inequality
\begin{displaymath}
\begin{array}{c}
\|x(\vartheta) - x'(\vartheta)\| + \|x_\vartheta(\cdot) - x'_\vartheta(\cdot)\|_1 \\[0.1cm]
+ \bigg|\displaystyle\int\limits_\tau^\vartheta f^0(\xi,x(\xi),x_\xi(\cdot),u(\xi),v(\xi)) \mathrm{d} \xi - \int\limits_\tau^\vartheta f^0(\xi,x'(\xi),x'_\xi(\cdot),u(\xi),v(\xi)) \mathrm{d} \xi\bigg| \\[0.6cm]
\leq \lambda_* \big(\|z - z'\| + \|w(\cdot) - w'(\cdot)\|_1\big).
\end{array}
\end{displaymath}
\end{lemma}

\begin{lemma}\label{lem:kappa}
Let $(z,w(\cdot)) \in [t_0,\vartheta] \times \mathrm{PC}$. Then, for every $\varepsilon > 0$, there exists $\delta = \delta(\varepsilon) > 0$ such that
\begin{displaymath}
\|\kappa_t(\cdot) - w(\cdot)\|_1 \leq \varepsilon,\quad t \in [\tau,\tau+\delta]\cap[t_0,\vartheta],\quad \kappa(\cdot) \in \Lambda_0(\tau,z,w(\cdot)),\quad \tau \in [t_0,\vartheta].
\end{displaymath}
\end{lemma}

%\begin{lemma}\label{lem:kappa}
%Let $(\tau,z,w(\cdot)) \in \mathbb G$ and $\kappa(\cdot) \in \Lambda_0(\tau,z,w(\cdot))$. Then, for every $\varepsilon > 0$, there exists $\delta = \delta(\varepsilon) > 0$ such that, for every $t,t' \in [\tau,\vartheta]$: $|t - t'| \leq \delta$, the estimate $\|\kappa_{t}(\cdot) - \kappa_{t'}(\cdot)\|_1 \leq \varepsilon$ holds.
%\end{lemma}

\section{Properties of lower and upper value functionals}
In this section, we prove auxiliary properties of the lower value functionals $\rho^u$ and the upper value functional $\rho^v$ defined by \cref{rho_u,rho_v}, respectively.

According terminology from \cite[Section 4.2]{Krasovskii_Subbotin_1988}, let us give the following definition.

\begin{definition}\label{def:stable}
A functional $\varphi \colon \mathbb G \mapsto \mathbb R$ is called stable (u- and v-stable) with respect to differential game \cref{dynamical_system,initial_condition,quality_index} if, for any $(\tau,z,w(\cdot)) \in \mathbb G$ and $t \in [\tau,\vartheta]$, it satisfies the estimates
\begin{displaymath}
\begin{array}{c}
\sup\limits_{v(\cdot) \in \mathcal{V}_\tau} \inf\limits_{u(\cdot) \in \mathcal{U}_\tau} \bigg(\varphi(t,x(t),x_t(\cdot)) + \displaystyle\int\limits_\tau^t f^0(\xi,x(\xi),x_\xi(\cdot),u(\xi),v(\xi)) \mathrm{d} \xi \bigg) \leq \varphi(\tau,z,w(\cdot)), \\[0.2cm]
\inf\limits_{u(\cdot) \in \mathcal{U}_\tau} \sup\limits_{v(\cdot) \in \mathcal{V}_\tau} \bigg(\varphi(t,x(t),x_t(\cdot)) + \displaystyle\int\limits_\tau^t f^0(\xi,x(\xi),x_\xi(\cdot),u(\xi),v(\xi)) \mathrm{d} \xi \bigg) \geq \varphi(\tau,z,w(\cdot)),
\end{array}
\end{displaymath}
where $x(\cdot) = x(\cdot\,|\,\tau,z,w(\cdot),u(\cdot),v(\cdot))$ is the motion of systems \cref{dynamical_system}.
\end{definition}

\begin{lemma}\label{lem:rho_is_stable}
The lower value functional $\rho^u$ and the upper value functional $\rho^v$ are stable with respect to differential game \cref{dynamical_system,initial_condition,quality_index}.
\end{lemma}
\begin{proof}
Similarly to \cite[Chapter VIII, Theorem 1.9]{Bardi_Capuzzo-Dolcetta_1997}, one can prove dynamic programming principle for $\rho^u$ and $\rho^v$, which implies the statement of the lemma.
\end{proof}

\begin{lemma}\label{lem:rho_in_Phi}
For the lower value functional $\rho^u$ and the upper value functional $\rho^v$, the inclusions $\rho^u, \rho^v \in \Phi$ hold.
\end{lemma}
\begin{proof}
We will prove the lemma only for $\rho^u$ . For $\rho^v$, the proof is similar.

First, let us show that $\rho^u$ is Lipschitz continuous (see \cref{phi_lipshiz_continuous}). Let $\alpha > 0$. In accordance with \cref{lem:alpha_x-lambda_x,cond:sigma_lipshiz_continuous,lem:motion_lipshiz_continuous}, define  $\alpha_X = \alpha_X(\alpha)$, $\lambda_* = \lambda_*(\alpha)$, and $\lambda_\sigma = \lambda_\sigma(\alpha_X)$, respectively. Put $\lambda_\varphi = \lambda_* (\lambda_\sigma + 1)$. Let $\tau \in [t_0,\vartheta]$, $(z,w(\cdot)), (z',w'(\cdot)) \in P(\alpha)$, and $\varepsilon > 0$. By definition of $\rho^u$ (see \ref{rho_u}), on the one hand, there exists a non-anticipative strategy $\hat{Q}^u_\tau$ such that
\begin{displaymath}
\sup\limits_{v(\cdot) \in \mathcal{V}_\tau} \gamma(\tau,z,w(\cdot),\hat{Q}^u_\tau,v(\cdot)) \leq \rho^u(\tau,z,w(\cdot)) + \varepsilon / 2,
\end{displaymath}
and, on the other hand, there exists $\hat{v}(\cdot) \in \mathcal{V}_\tau$ such that
\begin{displaymath}
\rho^u(\tau,z',w'(\cdot)) \leq \gamma(\tau,z',w'(\cdot),\hat{Q}^u_\tau,\hat{v}(\cdot)) + \varepsilon / 2.
\end{displaymath}
Therefore, taking into account \cref{quality_index} and the choice of $\lambda_\varphi$, we derive
\begin{displaymath}
\begin{array}{c}
\rho^u(\tau,z',w'(\cdot)) - \rho^u(\tau,z,w(\cdot)) \leq |\sigma(x'(\vartheta),x'_\vartheta(\cdot)) - \sigma(x(\vartheta),x_\vartheta(\cdot))| \\[0.1cm]
+ \displaystyle\int\limits_\tau^\vartheta |f^0(\xi,x'(\xi),x'_\xi(\cdot),\hat{u}(\xi),\hat{v}(\xi)) - f^0(\xi,x(\xi),x_\xi(\cdot),\hat{u}(\xi),\hat{v}(\xi))| \mathrm{d} \xi + \varepsilon \\[0.5cm]
\leq \lambda_\varphi\big(\|z' - z\| + \|w'(\cdot) - w(\cdot)\|_1\big) + \varepsilon.
\end{array}
\end{displaymath}
where $\hat{u}(\cdot) = \hat{Q}^u_\tau[\hat{v}(\cdot)](\cdot)$. Since this estimate holds for any $(z,w(\cdot)), (z',w'(\cdot)) \in P(\alpha)$, and $\varepsilon > 0$, we obtain \cref{phi_lipshiz_continuous}. Thus, Lipschitz continuity of $\rho^u$ is proved.

Let $(z,w(\cdot)) \in \mathbb R^n \times \mathrm{PC}$. Let us show that the function $[t_0,\vartheta] \ni \tau \mapsto \rho^u = \rho^u(\tau,z,w(\cdot)) \in \mathbb R$ is uniformly continuous. Let $\varepsilon > 0$. Put $\alpha = \max\{\|z\|,\|w(\cdot)\|_\infty\}$. In accordance with \cref{lem:alpha_x-lambda_x} and  Lipschitz continuous of $\rho^u$ (that is proved above), define  $\alpha_X = \alpha_X(\alpha)$, $\lambda_X = \lambda_X(\alpha)$, and $\lambda_\varphi = \lambda_\varphi(\alpha_X)$. Due to \cref{lem:kappa}, there exists $\delta_* > 0$ such that
\begin{displaymath}
\|\kappa_t(\cdot) - w(\cdot)\|_1 \leq \varepsilon/(6\lambda_\varphi),\ \ t \in [\tau,\tau+\delta_*] \cap [t_0,\vartheta],\ \ \kappa(\cdot) \in \Lambda_0(\tau,z,w(\cdot)),\ \ \tau \in [t_0,\vartheta].
\end{displaymath}
Taking the constant $c_f$ from \cref{cond:f_sublinear_growth}, put
\begin{displaymath}
\delta = \min\big\{1,\delta_*, \varepsilon / (6 \lambda_\varphi \lambda_X (1 + \vartheta - t_0)), \varepsilon / (3 c_f (1 + (2 + h) \alpha_X))\big\}
\end{displaymath}
Let $\tau,t \in [t_0,\vartheta]$ be such that $|\tau - t| \leq \delta$. Without loss of generality, suppose that $\tau \leq t$. Let $v(\cdot) \in \mathcal{V}_\tau$. According to the first estimate in \cref{def:stable} and \cref{lem:rho_is_stable}, there exists $u(\cdot) \in \mathcal{U}_\tau$ such that, for the motion $x(\cdot) = x(\cdot\,|\,\tau,z,w(\cdot),u(\cdot),v(\cdot))$ of system \cref{dynamical_system,initial_condition}, we have
\begin{equation}\label{lem:rho_in_Phi:stable}
\rho^u(t,x(t),x_t(\cdot)) + \int\limits_\tau^t f^0(\xi,x(\xi),x_\xi(\cdot),u(\xi),v(\xi)) \mathrm{d} \xi \leq \rho^u(\tau,z,w(\cdot)) + \varepsilon / 3.
\end{equation}
Due to the choice of $\lambda_X$, $\lambda_\varphi$, $\delta_*$, $\delta$, we derive
\begin{equation}\label{lem:rho_in_Phi:lipschitz_continuous}
\begin{array}{c}
|\rho^u(t,x(t),x_t(\cdot)) - \rho^u(t,z,w(\cdot))| \\[0.2cm]
\leq \lambda_\varphi\big(\|x(t) - z\| + \|x_t(\cdot) - \kappa_t(\cdot)\|_1 + \|\kappa_t(\cdot) - w(\cdot)\|_1\big) \leq \varepsilon / 3.
\end{array}
\end{equation}
According to the choice of $\alpha_X$, $c_f$, and $\delta$, we get
\begin{equation}\label{lem:rho_in_Phi:f_sublinear_growth}
\int\limits_\tau^t |f^0(\xi,x(\xi),x_\xi(\cdot),u(\xi),v(\xi))| \mathrm{d} \xi \leq c_f(1 + (2 + h) \alpha_X)(t - \tau) \leq \varepsilon / 3.
\end{equation}
From estimates \cref{lem:rho_in_Phi:stable,lem:rho_in_Phi:lipschitz_continuous,lem:rho_in_Phi:f_sublinear_growth}, we obtain $\rho^u(t,z,w(\cdot)) - \rho^u(\tau,z,w(\cdot)) \leq \varepsilon$. The inequality $\rho^u(t,z,w(\cdot)) - \rho^u(\tau,z,w(\cdot)) \geq  - \varepsilon$ can be proved in the similar way, using the second estimate in \cref{def:stable} instead of the first one.
\end{proof}

%%%%%%%%%%%%%%%%%%%%%%%%%%%%%%%%%%%%%%%%%%%%%%%%%%%%%%%%%%%%%%%%%%%%%%%%%%%%%%%%%%%%%%%%
%%%%%%%%%%%%%%%%%%%%%%%%%%%%%%%%%%%%%%%%%%%%%%%%%%%%%%%%%%%%%%%%%%%%%%%%%%%%%%%%%%%%%%%%
%%%%%%%%%%%%%%%%%%%%%%%%%%%%%%%%%%%%%%%%%%%%%%%%%%%%%%%%%%%%%%%%%%%%%%%%%%%%%%%%%%%%%%%%
%%%%%%%%%%%%%%%%%%%%%%%%%%%%%%%%%%%%%%%%%%%%%%%%%%%%%%%%%%%%%%%%%%%%%%%%%%%%%%%%%%%%%%%%
%%%%%%%%%%%%%%%%%%%%%%%%%%%%%%%%%%%%%%%%%%%%%%%%%%%%%%%%%%%%%%%%%%%%%%%%%%%%%%%%%%%%%%%%
%%%%%%%%%%%%%%%%%%%%%%%%%%%%%%%%%%%%%%%%%%%%%%%%%%%%%%%%%%%%%%%%%%%%%%%%%%%%%%%%%%%%%%%%

\section{Existence of viscosity solutions} In the section, we prove that the lower value functional $\rho^u$ and the upper value functional $\rho^v$, defined in \cref{rho_u,rho_v}, respectively, are viscosity solutions of problem \cref{Hamilton-Jacobi_equation,terminal_condition}, that implies the existence of viscosity solutions. Let us first prove auxiliary lemmas.

\begin{lemma}\label{lem:tilde_phi}
Let $\varphi \in \Phi$, $(\tau,z,w(\cdot)) \in \mathbb G$, and $\kappa(\cdot) \in \Lambda_0(\tau,z,w(\cdot))$. Then the function $\tilde{\varphi}$, defined by \cref{tilde_phi}, is continuous.
\end{lemma}
\begin{proof}
Let $(t,x),(t^i,x^i) \in [\tau,\vartheta] \times \mathbb R^n$, $i \in \mathbb N$, and $t^i \to t$, $x^i \to x$ as $i \to \infty$. Then, due to \cref{lem:kappa}, we have $\|\kappa_{t^i}(\cdot) - \kappa_t(\cdot)\|_1 \to 0$ as $i \to \infty$. Let $\alpha_* > 0$ be such that $(z,w(\cdot)) \in P(\alpha_*)$ and
$\|x\| \leq \alpha_*$, $\|x^i\| \leq \alpha_*$, $i \in \mathbb N$. Then, we obtain $(x,\kappa_t(\cdot)), (x^i,\kappa_{t^i}(\cdot)) \in P(\alpha_*)$, $i \in \mathbb N$. Thus, due to the inclusion $\varphi \in \Phi$, we conclude
\begin{displaymath}
\begin{array}{c}
|\tilde{\varphi}(t,x) - \tilde{\varphi}(t^i,x^i)| = |\varphi(t,x,\kappa_t(\cdot)) - \varphi(t^i,x^i,\kappa_{t^i}(\cdot))| \\[0.2cm]
\leq |\varphi(t,x,\kappa_t(\cdot)) - \varphi(t^i,x,\kappa_t(\cdot))|
+ \lambda_\varphi(\alpha_*) \big(\|x - x^i\| + \|\kappa_t(\cdot) - \kappa_{t^i}(\cdot)\|_1\big) \to 0,
\end{array}
\end{displaymath}
as $i \to \infty$.
\end{proof}

Denote
\begin{equation}\label{g}
g(\tau,z,w(\cdot),u,v,s) = \langle f(\tau,z,w(\cdot),u,v),s \rangle + f^0(\tau,z,w(\cdot),u,v).
\end{equation}

\begin{lemma}\label{lem:tilde_g}
Let $(\tau,z,w(\cdot)) \in \mathbb G$ and $\kappa(\cdot) \in \Lambda_0(\tau,z,w(\cdot))$. Let $\delta_w > 0$ be such that $w(\cdot)$ is continuous on $[-h,-h+\delta_w]$. Then the following function is continuous:
\begin{displaymath}
\tilde{g}(t,x,u,v,s) = g(t,x,\kappa_t(\cdot),u,v,s),\quad
(t,x,u,v,s) \in [\tau,\tau + \delta_w] \times \mathbb R^n \times \mathbb U \times \mathbb V \times \mathbb R^n.
\end{displaymath}
\end{lemma}
The proof is carried out similarly to \cref{lem:tilde_phi}, using \cref{cond:f_continuous,cond:f_lipshiz_continuous,cond:f_sublinear_growth}.

\begin{lemma}\label{lem:rho_is_viscosity_solution}
The lower value functional $\rho^u$ and the upper value functional $\rho^v$ are viscosity solutions of problem \cref{Hamilton-Jacobi_equation,terminal_condition}.
\end{lemma}
\begin{proof}
We will prove the lemma only for $\rho^u$. For $\rho^v$, the proof is similar.

By definition of $\rho^u$ (see \cref{rho_u}), it satisfies terminal condition \cref{terminal_condition}. Due to \cref{lem:rho_in_Phi}, the inclusion $\rho^u \in \Phi$ holds. Thus, in accordance with \cref{def:viscosity_solution} of a viscosity solution, in order to prove the lemma, it is sufficient to show that $\rho^u$ satisfies \cref{sub_viscosity_solution,super_viscosity_solution}.

For the sake of a contradiction, suppose that \cref{sub_viscosity_solution} is not valid for $\rho^u$. Then, taking into account \cref{cond:f_saddle_point}, definitions \cref{Hamiltonian,g}, and \cref{lem:tilde_g}, there exist $(\tau,z,w(\cdot)) \in \mathbb G$, $\tau < \vartheta$, $\psi \in \mathrm{C}^1(\mathbb R \times \mathbb R^n,\mathbb R)$, $v_0 \in \mathbb V$, $\delta_0 \in (0,h)$, and $\theta > 0$ such that
\begin{equation}\label{lem:rho_is_viscosity_solution:contradiction}
\begin{array}{c}
\rho^u(\tau,z,w(\cdot)) - \psi(\tau,z) \leq \rho^u(t,x,\kappa_t(\cdot)) - \psi(t,x), \\[0.2cm]
\partial \psi(t,x) / \partial \tau + \min\limits_{u \in \mathbb U} g(t,x,\kappa_t(\cdot),u,v_0,\nabla_z \psi(t,x)) \geq \theta,
\end{array}
\quad (t, x) \in O^+_{\delta_0}(\tau, z).
\end{equation}
Let $\alpha = \max\{\|z\|,\|w(\cdot)\|_\infty\}$. According to \cref{lem:alpha_x-lambda_x}, \cref{cond:f_lipshiz_continuous}, and the inclusion $\rho^u \in \Phi$, define $\alpha_X = \alpha_X(\alpha) > 0$, $\lambda_X = \lambda_X(\alpha) > 0$, $\lambda_f = \lambda_f(\alpha_X) > 0$, and $\lambda_\varphi = \lambda_\varphi(\alpha_X) > 0$. Denote
\begin{equation}\label{lem:rho_is_viscosity_solution:c_1_delta_1}
c_1 = \max\limits_{(t, x) \in O^+_{\delta_0}(\tau, z)} \nabla_z\psi(t,x),\quad
\delta_1 = \min\bigg\{\frac{\delta_0}{1 + \lambda_X},\frac{\theta}{8\lambda_\varphi\lambda_X},\frac{\sqrt{\theta}}{\sqrt{2\lambda_f (1+c_1)}\lambda_X}\bigg\}.
\end{equation}
Then, for every $u(\cdot) \in \mathcal{U}_\tau$ and $v(\cdot) \in \mathcal{V}_\tau$, the motion $x(\cdot) = x(\cdot\,|\,\tau,z,w(\cdot),u(\cdot),v(\cdot))$ of system \cref{dynamical_system,initial_condition} satisfies the estimates
\begin{displaymath}
\begin{array}{c}
(x(t),x_t(\cdot)),(x(t),\kappa_t(\cdot)) \in P(\alpha_X),\quad |x(t) - z| \leq \lambda_X |t - \tau| \leq \delta_0,\\[0.2cm]
\|x_t(\cdot) - \kappa_t(\cdot)\|_1 \leq \lambda_X (t - \tau)^2 \leq \min\bigg\{\displaystyle\frac{\delta_1\theta}{8 \lambda_\varphi},\frac{\theta}{2\lambda_f(1 + c_1)}\bigg\},
\end{array}
\quad t \in [\tau,\tau + \delta_1],
\end{displaymath}
wherefrom, taking into account the choice of $\lambda_\varphi$, $\lambda_f$, and $c_1$, together with definition \cref{g} of $g$ and the equality $x_t(-h) = \kappa_t(-h) = w(t-\tau-h)$, $t \in [\tau,\tau+h)$, we derive
\begin{displaymath}
\begin{array}{c}
(t,x(t)) \in O^+_{\delta_0}(\tau, z),\quad |\rho^u(t,x(t),x_t(\cdot)) - \rho^u(t,x(t),\kappa_t(\cdot))| \leq \delta_1 \theta / 8,\\[0.2cm]
\!|g(t,x(t),x_t(\cdot),u(t),v_0,\nabla_z\psi(t,x(t))) - g(t,x(t),\kappa_t(\cdot),u(t),v_0,\nabla_z\psi(t,x(t)))| \leq \theta / 2.
\end{array}
\end{displaymath}
Hence, from \cref{lem:rho_is_viscosity_solution:contradiction}, we obtain
\begin{equation}\label{lem:rho_is_viscosity_solution:contradiction_2}
\begin{array}{c}
\rho^u(\tau,z,w(\cdot)) - \psi(\tau,z) \leq \rho^u(t,x(t),x_t(\cdot)) - \psi(t,x(t)) + \delta_1 \theta / 8, \\[0.2cm]
\partial \psi(t,x(t)) / \partial \tau + g(t,x(t),x_t(\cdot),u(t),v_0,\nabla_z\psi(t,x(t))) \geq \theta/2,
\end{array}
\ \ t \in [\tau,\tau + \delta_1].
\end{equation}
Due to \cref{lem:rho_is_stable} and the first estimate in \cref{def:stable}, for $t = \tau + \delta_1$,  there exists $u(\cdot) \in \mathcal{U}_\tau$ such that the motion $x(\cdot) = x(\cdot\,|\,\tau,z,w(\cdot),u(\cdot),v(\cdot) \equiv v_0)$ of system \cref{dynamical_system,initial_condition} satisfies the inequality
\begin{equation}\label{lem:rho_is_viscosity_solution:stable}
\rho^u(t,x(t),x_t(\cdot)) + \int\limits_\tau^t f^0(\xi,x(\xi),x_\xi(\cdot),u(\xi),v_0)\mathrm{d} \xi \leq \rho^u(\tau,z,w(\cdot)) + \delta_1 \theta / 8.
\end{equation}
%For the function $\psi$, in accordance with \cref{dynamical_system}, the formula below holds:
Using continuous differentiability of $\psi$, taking into account \cref{dynamical_system}, we derive
\begin{displaymath}
\psi(t,x(t)) - \psi(\tau,z) = \! \int\limits_\tau^t\!\bigg(\frac{\partial \psi(\xi,x(\xi))}{\partial \tau} + \langle f(\xi,x(\xi),x_\xi(\cdot),u(\xi),v_0),\nabla_z\psi(\xi,x(\xi))\rangle\bigg)\mathrm{d}\xi.
\end{displaymath}
Using this formula together with inequalities \cref{lem:rho_is_viscosity_solution:contradiction_2,lem:rho_is_viscosity_solution:stable}, we get the contradictory inequality $\theta/2 \leq \theta / 4$. Thus, \cref{sub_viscosity_solution} is valid for $\rho^u$. In the same way we can prove that \cref{super_viscosity_solution} is valid for $\rho^u$.
\end{proof}

%%%%%%%%%%%%%%%%%%%%%%%%%%%%%%%%%%%%%%%%%%%%%%%%%%%%%%%%%%%%%%%%%%%%%%%%%%%%%%%%%%%%%%%%
%%%%%%%%%%%%%%%%%%%%%%%%%%%%%%%%%%%%%%%%%%%%%%%%%%%%%%%%%%%%%%%%%%%%%%%%%%%%%%%%%%%%%%%%
%%%%%%%%%%%%%%%%%%%%%%%%%%%%%%%%%%%%%%%%%%%%%%%%%%%%%%%%%%%%%%%%%%%%%%%%%%%%%%%%%%%%%%%%
%%%%%%%%%%%%%%%%%%%%%%%%%%%%%%%%%%%%%%%%%%%%%%%%%%%%%%%%%%%%%%%%%%%%%%%%%%%%%%%%%%%%%%%%
%%%%%%%%%%%%%%%%%%%%%%%%%%%%%%%%%%%%%%%%%%%%%%%%%%%%%%%%%%%%%%%%%%%%%%%%%%%%%%%%%%%%%%%%
%%%%%%%%%%%%%%%%%%%%%%%%%%%%%%%%%%%%%%%%%%%%%%%%%%%%%%%%%%%%%%%%%%%%%%%%%%%%%%%%%%%%%%%%

\section{Uniqueness of the viscosity solution} In the section, we prove the uniqueness of the viscosity solution. The proof mainly follows the scheme from \cite[Chapter III, Theorem 3.15]{Bardi_Capuzzo-Dolcetta_1997}, but has some specifics related to the functional position space $\mathbb G$.

\begin{lemma}\label{lem:zeta_star}
Let the Hamiltonian $H$ be defined by \cref{Hamiltonian}. One can choose $\zeta_* \in (0, h)$ such that the following statement holds. Let $(\tau_*,z_*,w_*(\cdot)) \in \mathbb G$, $\tau_* < \vartheta$, and $\nu_* \in (\tau_*,\min\{\tau_* + \zeta_*, \vartheta\}]$. Then there exists $c_* > 0$ such that, if
\begin{equation}\label{lem:zeta_star:condition}
t \in [\tau_*,\nu_*],\quad (x,r(\cdot)) \in P(\alpha_*),\quad
\alpha_* = (\nu_* - \tau_*) c_* + 1 + \max\{\|z_*\|,\|w_*(\cdot)\|_\infty\},
\end{equation}
where $P(\alpha_*)$ is defined according to \cref{P}, then
\begin{equation}\label{lem:c_star:statement}
|H(t,x,r(\cdot),s) - H(t,x,r(\cdot),s')| \leq c_*\|s - s'\|,\quad s, s' \in \mathbb R^n.
\end{equation}
\end{lemma}
\begin{proof}
Taking the constant $c_f$ from \cref{cond:f_sublinear_growth}, put
\begin{equation}\label{lem:c_star:zeta_star}
\zeta_* =  \min\{1 / (2 c_f (3 + h)),h/2\}.
\end{equation}
Let $(\tau_*,z_*,w_*(\cdot)) \in \mathbb G$, $\tau_* < \vartheta$, and $\nu_* \in (\tau_*,\min\{\tau_* + \zeta_*, \vartheta\}]$. Put
\begin{equation}\label{lem:c_star:c_star}
c_* = 2 c_f (3 + h) \big(1 +\max\{\|z_*\|,\|w_*(\cdot)\|_\infty\}\big).
\end{equation}
Let $t \in [\tau_*,\nu_*]$ and $(x,r(\cdot)) \in P(\alpha_*)$. Then, due to \cref{cond:f_sublinear_growth}, definition \cref{Hamiltonian} of $H$, and the choice of $\alpha_*$ in \cref{lem:zeta_star:condition}, we derive
\begin{displaymath}
\begin{array}{c}
|H(t,x,r(\cdot),s) - H(t,x,r(\cdot),s')| \leq c_f \big(1 + \|x\| + \|r(\cdot)\|_1 + \|r(-h)\|\big) \|s - s'\|\\[0.2cm]
\leq c_f (3 + h) \alpha_* \|s - s'\| \leq c_f (3 + h) (\zeta_* c_* + 1 + \max\{\|z_*\|,\|w_*(\cdot)\|_\infty\}) \|s - s'\|,\\[0.2cm]
\end{array}
\end{displaymath}
wherefrom, using \cref{lem:c_star:zeta_star,lem:c_star:c_star}, we obtain \cref{lem:c_star:statement}.
\end{proof}

Denote by $\mathrm{PC}_0$ the set of continuous functions $x(\cdot) \colon [-h,0) \mapsto \mathbb R^n$, which have a finite left limits at the point $0$. Then, due to definition \cref{PC_G} of $\mathrm{PC}$, we have $\mathrm{PC}_{0} \subset \mathrm{PC}$.

\begin{lemma}\label{lem:taus_zs_ws_nus}
Let $\Delta \varphi \in \Phi$ satisfy the relations
\begin{equation}\label{lem:taus_zs_ws_nus:condition}
\Delta \varphi \not\equiv 0,\quad \Delta \varphi(\vartheta,z,w(\cdot)) = 0,\quad (z,w(\cdot)) \in \mathbb R^n \times \mathrm{PC},
\end{equation}
where the first relation means that there exists a point $(\tau,z,w(\cdot)) \in \mathbb G$ such that $\Delta \varphi (\tau,z,w(\cdot)) \neq 0$.
Let $\zeta_* > 0$. Then there exist $(\tau_*,z_*,w_*(\cdot)) \in \mathbb G$, $\tau_* < \vartheta$, $w_*(\cdot) \in \mathrm{PC}_0$, and $\nu_* \in (\tau_*, \min\{\tau_* + \zeta_*, \vartheta\}]$, $\theta_* > 0$ such that
\begin{equation}\label{lem:taus_zs_ws_nus:statement}
|\Delta \varphi(\tau_*,z_*,w_*(\cdot))| > \theta_*,\quad \Delta\varphi(\nu_*,x,r(\cdot)) = 0,\quad (x,r(\cdot)) \in \mathbb R^n \times \mathrm{PC}.
\end{equation}
\end{lemma}
\begin{proof}
Let $k \in \mathbb N$ be such that $\Delta t = (\vartheta - t_0) / k < \zeta_*$. Denote $\nu_i = t_0 + i \Delta t$, $i \in \overline{0,k}$, and
\begin{displaymath}
I = \big\{i \in \overline{0,k} \colon \Delta \varphi(t,x,w(\cdot)) = 0,\, (t,x,w(\cdot)) \in [\nu_i, \vartheta] \times \mathbb R^n \times \mathrm{PC} \big\}.
\end{displaymath}
Due to \cref{lem:taus_zs_ws_nus:condition}, we have $I \neq \emptyset$ (at least $k \in I$) and $i_0 = \min I > 0$. Hence, there exists $(\tau_*,z_*,w'_*(\cdot)) \in [\nu_{i_0 - 1}, \nu_{i_0}) \times \mathbb R^n \times \mathrm{PC}$ such that $\Delta\varphi(\tau_*,z_*,w'_*(\cdot)) \neq 0$. Defining $\nu_* = \nu_{i_0}$, we obtain the second relation in \cref{lem:taus_zs_ws_nus:statement}.

Let $\theta_* = |\Delta\varphi(\tau_*,z_*,w'_*(\cdot))| / 3$. Then, due to the inclusion $\Delta \varphi \in \Phi$ (see \cref{phi_lipshiz_continuous}), for
$\alpha = \|w'_*(\cdot)\|_\infty$, there exists $\lambda_\varphi = \lambda_\varphi(\alpha) > 0$ such that
\begin{equation}\label{lem:taus_zs_ws_nus:lipshitz}
\begin{array}{c}
|\Delta \varphi(\tau_*,z_*,w(\cdot))| \geq 3 \theta_* - |\Delta \varphi(\tau_*,z_*,w'_*(\cdot)) - \Delta \varphi(\tau_*,z_*,w(\cdot))|\\[0.2cm]
> 2 \theta_* - \lambda_\varphi\|w'_*(\cdot) - w(\cdot)\|_1,\quad (z_*,w(\cdot)) \in P(\alpha).
\end{array}
\end{equation}
By the function $w'_*(\cdot) \in \mathrm{PC}$, one can find (see, e.g., \cite[p. 214]{Natanson_1960}) a function $w_*(\cdot) \in \mathrm{PC}_0$ such that $\|w_*(\cdot)\|_\infty \leq \alpha$ and $\|w'_*(\cdot) - w_*(\cdot)\|_1 \leq \theta_* / \lambda_\varphi$. Then, using \cref{lem:taus_zs_ws_nus:lipshitz} for $w(\cdot) = w_*(\cdot)$, we obtain the first relation in \cref{lem:taus_zs_ws_nus:statement}.
\end{proof}

For $(\tau,z,w(\cdot)) \in \mathbb G$, $\nu \in [\tau,\vartheta]$, $c > 0$, and $\beta \geq 0$, denote
\begin{equation}\label{Omega_Pi}
\hspace{-0.3cm}
\begin{array}{rcl}
\Omega(\tau,z,\nu,c,\beta) \!\!\!\!\! & = & \!\!\!\!\! \big\{(t,x) \in [\tau,\nu] \times \mathbb R^n \colon \|x - z\| \leq (t - \tau) c + \beta\big\},\\[0.2cm]
\Pi(\tau,z,w(\cdot),\nu,c) \!\!\!\!\! & = & \!\!\!\!\! \big\{y(\cdot) \in \mathrm{PC}([\tau-h,\vartheta], \mathbb R^n) \colon y(t) = w(t - \tau),\, t \in [\tau-h,\tau),\\[0.2cm]
&&\hspace{1.7cm} y(\tau) = z,\ (t,y(t)) \in \Omega(\tau,z,\nu,c,0),\, t \in [\tau,\nu]\big\}.
\end{array}
\end{equation}

\begin{lemma}\label{lem:tau_z_nu}
Let $\Delta \varphi \in \Phi$, $(\tau_*,z_*,w_*(\cdot)) \in \mathbb G$, $\tau_* < \vartheta$, $\nu_* \in (\tau_*, \vartheta]$, and $\theta_* > 0$ satisfy the relations
\begin{equation}\label{lem:tau_z_nu:condition}
\Delta \varphi(\tau_*,z_*,w_*(\cdot))  > \theta_*,\quad \Delta \varphi(\nu_*,x,r(\cdot)) = 0,\quad (x,r(\cdot)) \in \mathbb R^n \times \mathrm{PC}.
\end{equation}
Let $c_*, \zeta > 0$. Denote
\begin{displaymath}
\Omega_* = \Omega(\tau_*,z_*,\nu_*,c_*,0),\quad \Pi_* = \Pi(\tau_*,z_*,w_*(\cdot),\nu_*,c_*).
\end{displaymath}
Then there exist $(\tau,z) \in \Omega_*$, $y^*(\cdot) \in \Pi_*$, and $\nu \in (\tau,\min\{\tau + \zeta,\nu_*\}]$ such that, for every $x \in \mathbb R^n$ satisfying the inclusion $(\nu,x) \in \Omega(\tau,z,\nu,c_*,0)$, and for every $y(\cdot) \in \Pi(\tau,z,y^*_\tau(\cdot),\nu,c_*)$, the following inequality holds:
\begin{equation}\label{lem:tau_z_nu:statement}
\Delta\varphi(\tau,z,y^*_\tau(\cdot)) > \theta_* (\nu - \tau)/(\nu_* - \tau_*)  + \Delta\varphi(\nu,x,y_\nu(\cdot)).
\end{equation}
\end{lemma}
\begin{proof}
Aiming for a contradiction, suppose that, for each $(\tau,z) \in \Omega_*$, $y^*(\cdot) \in \Pi_*$, and $\nu \in (\tau,\min\{\tau + \zeta,\nu_*\}]$, there exist $x \in \mathbb R^n$ and $y(\cdot) \in \Pi(\tau,z,y^*_\tau(\cdot),\nu,c_*)$ such that
\begin{equation}\label{lem:tau_z_nu:contradiction}
(\nu, x) \in \Omega(\tau,z,\nu,c_*,0),\ \
\Delta\varphi(\tau, z, y^*_\tau(\cdot)) \leq \theta_* (\nu - \tau)/(\nu_* - \tau_*) + \Delta\varphi(\nu, x, y_\nu(\cdot)).
\end{equation}
Let $k \in \mathbb N$ be such that $\Delta \nu = (\nu_* - \tau_*) / k < \zeta$. Denote $\nu_i = \tau_* + i \Delta \nu$, $i \in \overline{0,k}$. Using mathematical induction, for each $i \in \overline{0,k}$, we will find $x_i \in \mathbb R^n$ and $y^i(\cdot) \in \Pi_*$, satisfying the relations below:
\begin{equation}\label{lem:tau_z_nu:induction}
(\nu_i,x_i) \in \Omega_*,\quad \theta_* (\nu_* - \nu_i) / (\nu_* - \tau_*) < \Delta\varphi(\nu_i,x_i,y^i_{\nu_i}(\cdot)).
\end{equation}
For $i = 0$, these relations follow from the first relation in \cref{lem:tau_z_nu:condition} taking $x_0 = z_*$ and an arbitrary $y^0(\cdot) \in \Pi_*$. Assume that \cref{lem:tau_z_nu:induction} holds for $i = l$. Due to \cref{lem:tau_z_nu:contradiction}, there exist $x_{l+1} \in \mathbb R^n$ and $\overline{y}^{\,l+1}(\cdot) \in \Pi(\nu_l,x_l,y^l_{\nu_l}(\cdot),\nu_{l+1},c_*)$ such that
\begin{equation}\label{lem:tau_z_nu:contradiction_l}
\begin{array}{c}
(\nu_{l+1},x_{l+1}) \in \Omega(\nu_l,x_l,\nu_{l+1},c_*,0),\\[0.2cm]
\Delta\varphi(\nu_l,x_l,y^l_{\nu_l}(\cdot)) \leq \theta_* (\nu_{l+1} - \nu_l)/(\nu_* - \tau_*) + \Delta\varphi(\nu_{l+1},x_{l+1},\overline{y}^{\,l+1}_{\nu_{l+1}}(\cdot)).
\end{array}
\end{equation}
Define $y^{l+1}(\cdot)$ by $y^{l+1}(t) = y^l(t)$ for $t \in [\tau_*,\tau_l)$, and $y^{l+1}(t) = \overline{y}^{\,l+1}(t)$ for $t \in [\tau_l,\vartheta]$. Then we have $y^{l+1}(\cdot) \in \Pi_*$. Hence, using \cref{lem:tau_z_nu:induction} for $i = l$ and \cref{lem:tau_z_nu:contradiction_l}, we obtain \cref{lem:tau_z_nu:induction} for $i = l + 1$. It means that \cref{lem:tau_z_nu:induction} is valid for any $i \in \overline{0,k}$, but \cref{lem:tau_z_nu:induction} for $i = k$ contradicts the second relation in \cref{lem:tau_z_nu:condition}.
\end{proof}

\begin{lemma}\label{lem:Omega}
Let the Hamiltonian $H$ be defined by \cref{Hamiltonian}. Let $\varphi_1, \varphi_2 \in \Phi$ be different functionals satisfying terminal condition \cref{terminal_condition}.
Then there exist $(\tau,z,w(\cdot)) \in \mathbb G$, $\tau < \vartheta$, $\nu \in (\tau,\vartheta]$, and $c_*, \beta, \theta, \alpha_\varphi, \lambda_\varphi > 0$ such that the functions
\begin{equation}\label{lem:Omega:definitions}
\begin{array}{c}
\tilde{H}(t,x,s) = H(t,x,\kappa_t(\cdot),s),\quad \tilde{\varphi}_i(t,x) = \varphi_i(t,x,\kappa_t(\cdot)),\quad i = 1,2,\\[0.2cm]
(t,x) \in \Omega^\beta = \Omega(\tau,z,\nu,c_*,\beta),\quad \kappa(\cdot) \in \Lambda_0(\tau,z,w(\cdot)),\quad s\in \mathbb R^n,
\end{array}
\end{equation}
are continuous, the Hamiltonian $H$ satisfies the estimate
\begin{equation}\label{lem:Omega:statement_1}
|H(t,x,\kappa_t(\cdot),s) - H(t,x,\kappa_t(\cdot),s')| \leq c_* \|s - s'\|,\quad (t,x) \in \Omega^\beta,\quad s,s' \in \mathbb R^n,
\end{equation}
and the functionals $\varphi_i$, $i = 1,2$, satisfy the inequalities
\begin{equation}\label{lem:Omega:statement_2}
\begin{array}{c}
|\varphi_i(t,x,\kappa_t(\cdot))| \leq \alpha_\varphi,\quad |\varphi_i(t,x,\kappa_t(\cdot)) - \varphi_i(t,x',\kappa_t(\cdot))| \leq \lambda_\varphi \|x - x'\|,\\[0.2cm]
(t,x),(t,x') \in \Omega^\beta,
\end{array}
\end{equation}
the estimate
\begin{equation}\label{lem:Omega:statement_3}
\begin{array}{c}
|\varphi_i(t,x,\kappa_t(\cdot)) - \varphi_i(t,x,\kappa'_t(\cdot))| \leq (t - t') \theta,\\[0.2cm]
(t,x) \in ([t',\nu] \times \mathbb R^n) \cap \Omega^\beta,\quad \kappa'(\cdot) \in \Lambda_0(t',x',\kappa_{t'}(\cdot)),\quad (t',x') \in \Omega^\beta,
\end{array}
\end{equation}
and the inequality
\begin{equation}\label{lem:Omega:statement_4}
\begin{array}{c}
\varphi_1(\tau,z,w(\cdot)) - \varphi_2(\tau,z,w(\cdot)) \\[0.2cm]
> 4 (\nu - \tau) \theta + \varphi_1(\nu,x,\kappa_\nu(\cdot)) - \varphi_2(\nu,x,\kappa_\nu(\cdot)),
\end{array}
\quad (\nu,x) \in \Omega^\beta.
\end{equation}
\end{lemma}
\begin{proof}
Define the functional
\begin{equation}\label{lem:Omega:Delta_phi}
\Delta \varphi(\tau,z,w(\cdot)) = \varphi_1(\tau,z,w(\cdot)) - \varphi_2(\tau,z,w(\cdot)),\quad (\tau,z,w(\cdot)) \in \mathbb G.
\end{equation}
Then this functional satisfies the inclusion $\Delta \varphi \in \Phi$ (see \cref{phi_lipshiz_continuous}) and relations \cref{lem:taus_zs_ws_nus:condition}.
Let us take $\zeta_* \in (0, h)$ according to \cref{lem:zeta_star}. Due to \cref{lem:taus_zs_ws_nus}, there exist $(\tau_*,z_*,w_*(\cdot)) \in \mathbb G$, $\tau_* < \vartheta$, $w_*(\cdot) \in \mathrm{PC}_0$, $\nu_* \in (\tau_*, \min\{\tau_* + \zeta_*, \vartheta\}]$, and $\theta_* > 0$ such that
relations \cref{lem:taus_zs_ws_nus:statement} hold. Then, without loss of generality, we can suppose that \cref{lem:tau_z_nu:condition} holds. By \cref{lem:zeta_star}, there exists $c_* > 0$ such that \cref{lem:c_star:statement} is valid. Define $\alpha_*$ in accordance with \cref{lem:zeta_star:condition}. Then, due to the inclusion $\varphi_1,\varphi_2 \in \Phi$ (see \cref{phi_lipshiz_continuous}), there exists $\lambda_\varphi = \lambda_\varphi(\alpha_*)$ such that
\begin{equation}\label{lem:Omega:lambda_phi}
|\varphi_i(t,x,r(\cdot)) - \varphi_i(t,x',r'(\cdot))| \leq \lambda_\varphi\big(\|x - x'\| + \|r(\cdot) - r'(\cdot)\|_1\big)
\end{equation}
for any $t \in [t_0,\vartheta]$ and $(x,r(\cdot)), (x',r'(\cdot)) \in P(\alpha_*)$. Put
\begin{equation}\label{lem:Omega:theta_zeta}
\theta = \theta_* / (6 (\nu_* - \tau_*)),\quad \zeta = \min\{h + \tau_* - \nu_*, \theta / ((c_* + 1) \lambda_\varphi)\}.
\end{equation}
In accordance with \cref{lem:tau_z_nu}, let us choose $(\tau,z) \in \Omega(\tau_*,z_*,\nu_*,c_*,0)$, $y^*(\cdot) \in \Pi(\tau_*,z_*,w_*(\cdot),\nu_*,c_*)$, and $\nu \in (\tau,\min\{\tau + \zeta,\nu_*\}]$ such that \cref{lem:tau_z_nu:statement} holds. Put
\begin{equation}\label{lem:Omega:w_beta}
w(\cdot) = y^*_\tau(\cdot),\quad \beta = \min\{1,\zeta,(\nu - \tau) \zeta\}.
\end{equation}
Note that thus we have
\begin{equation}\label{lem:Omega:times}
-h < -h + \zeta \leq \tau_* - \nu_* < \tau_* - \tau \leq 0,\quad \nu \leq \tau + \zeta,
\end{equation}
and
\begin{equation}\label{lem:Omega:w}
w(\xi) =
\left\{
\begin{array}{ll}
y^*(\tau + \xi) \in \Omega(\tau_*,z_*,\nu_*,c_*,0),& \xi \in [\tau_* - \tau, 0],\\[0.2cm]
w_*(\tau - \tau_* + \xi),& \xi \in [- h, \tau_* - \tau).
\end{array}
\right.
\end{equation}
In accordance with \cref{lem:Omega:definitions}, define $\kappa(\cdot)$, $\Omega^\beta$, $\tilde{H}$, and $\tilde{\varphi}_i$.

Due to the inclusion $w_*(\cdot) \in \mathrm{PC}_0$, taking into account \cref{lem:Omega:times,lem:Omega:w}, the function $w(\cdot)$ is continuous on $[-h, -h + \zeta]$. Then, using \cref{lem:tilde_g} for $\delta_w = \zeta$, definition \cref{Hamiltonian} of $H$, and the second relation in \cref{lem:Omega:times}, we obtain continuity of $\tilde{H}$.

Since $(\tau,z) \in \Omega(\tau_*,z_*,\nu_*,c_*,0)$, using the choice of $\alpha_*$ and $\beta$, we have
\begin{equation}\label{lem:Omega:alpha_star}
\begin{array}{c}
\|x\| \leq (t - \tau) c_* + \beta + \|z\| \leq (t - \tau_*) c_* + 1 + \|z_*\| \leq \alpha_*,\\[0.2cm]
\|\kappa_t(\cdot)\|_\infty \! \leq \max\{\|z\|,\|w(\cdot)\|_\infty\}
\leq \max\{(\tau - \tau_*) c_* + \|z_*\|, \|w_*(\cdot)\|_\infty\} \leq \alpha_*
\end{array}
\end{equation}
for any $(t,x) \in \Omega^\beta$. Hence (see \cref{P}), we derive
\begin{equation}\label{lem:Omega:P}
(x,\kappa_t(\cdot)) \in P(\alpha_*),\quad (t,x) \in \Omega^\beta.
\end{equation}
Then, by the choice of $c_*$, we obtain \cref{lem:Omega:statement_1}.

Due to the inclusion $\varphi_1,\varphi_2 \in \Phi$ and \cref{lem:tilde_phi}, the functions $\tilde{\varphi}_1$ and $\tilde{\varphi}_2$ are continuous. Then, defining $\alpha_\varphi = \max\{|\varphi_i(t,x)|\,|\, (t,x) \in\Omega^\beta, i = 1,2\}$, we get the first inequality in \cref{lem:Omega:statement_2}. The second inequality in \cref{lem:Omega:statement_2} follows from \cref{lem:Omega:lambda_phi,lem:Omega:P}.

Let $(t',x') \in \Omega^\beta$, $\kappa'(\cdot) \in \Lambda_0(t',x',\kappa_{t'}(\cdot))$, and $(t,x) \in ([t',\nu] \times \mathbb R^n) \cap \Omega^\beta$. Then, using \cref{lem:Omega:w_beta,lem:Omega:times}, we have $\|x' - z\| \leq (t' - \tau) c_* + \beta \leq (\nu - \tau) c_* + \zeta \leq (c_* + 1) \zeta$, and, taking into account \cref{lem:Omega:theta_zeta}, we derive
\begin{equation}\label{lem:Omega:kappa_kappa}
\|\kappa_t(\cdot) - \kappa'_t(\cdot)\|_1 = (t - t') \|z - x'\| \leq (t - t') (c_* + 1) \zeta \leq (t - t') \theta / \lambda_\varphi.
\end{equation}
Due to \cref{lem:Omega:alpha_star}, we have $\|x\| \leq \alpha_*$ and $\|\kappa'_t(\cdot)\|_\infty\! \leq \max\{\|x'\|,\|\kappa_{t'}(\cdot)\|_\infty\}\! \leq \alpha_*$, that means the inclusion $(x,\kappa'_t(\cdot)) \in P(\alpha_*)$. Then, from \cref{lem:Omega:lambda_phi,lem:Omega:kappa_kappa}, we obtain \cref{lem:Omega:statement_3}.

From \cref{lem:Omega:statement_2,lem:Omega:Delta_phi,lem:Omega:w_beta,lem:Omega:theta_zeta}, we derive
\begin{displaymath}
\begin{array}{c}
\min\limits_{(\nu,x') \in \Omega(\tau,z,\nu,c_*,0)} |\Delta\varphi(\nu,x,\kappa_\nu(\cdot)) - \Delta\varphi(\nu,x',\kappa_\nu(\cdot))|\\[0.2cm]
\leq 2 \lambda_\varphi \min\limits_{(\nu,x') \in \Omega(\tau,z,\nu,c_*,0)} \|x - x'\|\leq 2 \lambda_\varphi \beta \leq 2 \lambda_\varphi (\nu - \tau) \zeta \leq 2 (\nu - \tau) \theta,
\end{array}
\quad (\nu,x) \in \Omega^\beta,
\end{displaymath}
wherefrom, using \cref{lem:tau_z_nu:statement,lem:Omega:theta_zeta}, we obtain \cref{lem:Omega:statement_4}.
\end{proof}

\begin{lemma}\label{lem:eta}
Let $(\tau,z,w(\cdot)) \in \mathbb G$, $\tau < \vartheta$, $\nu \in (\tau,\vartheta]$, and $c_*, \beta, \theta, \alpha_\varphi > 0$. Let $\Omega^\beta$ be defined by \cref{lem:Omega:definitions}. Then there exists $\eta \in \mathrm{C}^1(\mathbb R \times \mathbb R^n, \mathbb R)$ satisfying the relations
\begin{equation}\label{lem:eta:statement_1}
\begin{array}{c}
\eta(\tau,z) = 0,\quad \eta(t,x) \geq 0,\quad (t,x) \in \Omega^\beta,\\[0.2cm]
\eta(t,x) \geq 2 \alpha_\varphi + 4 (\nu - \tau) \theta,\quad (t,x) \in \partial\Omega^\beta,
\end{array}
\end{equation}
where $\partial\Omega^\beta = \{(t,x) \in \Omega^\beta \colon \|x - z\| = (t - \tau) c_* + \beta\}$, and the inequality
\begin{equation}\label{lem:eta:statement_2}
\|\nabla_x \eta(t,x)\| \leq - (1 / c_*) \partial \eta(t,x) / \partial t,\quad (t,x) \in \Omega^\beta.
\end{equation}
Moreover, for this function, there exists $\lambda_\eta > 0$ such that
\begin{equation}\label{lem:eta:statement_3}
|\eta(t,x) - \eta(t,x)| \leq \lambda_\eta \|x - x'\|,\quad (t,x),(t,x') \in \Omega^\beta.
\end{equation}
\end{lemma}
\begin{proof}
Let $\mu(\cdot) \colon \mathbb R \mapsto \mathbb R$ be a continuously differentiable function such that $\mu(l) \geq 0$, $\dot{\mu}(l) \geq 0$, $l \in \mathbb R$, and $\mu(\beta/2) = 0$, $\mu(\beta) = 2 \alpha_\varphi + 4 (\nu - \tau) \theta$. Put
\begin{displaymath}
\eta(t,x) = \mu(\tilde{\eta}(t,x)),\quad \tilde{\eta}(t,x) = \sqrt{\|x - z\|^2 + \beta^2 / 4} - (t - \tau) c_*,\quad (t,x) \in \mathbb R \times \mathbb R^n.
\end{displaymath}
Then this function satisfies the inclusion $\eta \in \mathrm{C}^1(\mathbb R \times \mathbb R^n, \mathbb R)$ and relations \cref{lem:eta:statement_1,lem:eta:statement_2}. Taking $\lambda_\eta = \max\limits_{(t,x) \in \Omega^\beta} \|\nabla_x \eta(t,x)\|$, we obtain \cref{lem:eta:statement_3}.
\end{proof}

\begin{lemma}\label{uniqueness_of_viscosity solution}
The viscosity solution of problem \cref{Hamilton-Jacobi_equation,terminal_condition} is unique.
\end{lemma}
\begin{proof}
Aiming for a contradiction, suppose that $\varphi_1$ and $\varphi_2$ are different viscosity solutions of problem \cref{Hamilton-Jacobi_equation,terminal_condition}. According to \cref{lem:Omega}, define $(\tau,z,w(\cdot)) \in \mathbb G$, $\tau < \vartheta$, $\nu \in (\tau,\vartheta]$, and $c_*, \beta, \theta, \alpha_\varphi, \lambda_\varphi > 0$ such that the functions $\tilde{H}$ and $\tilde{\varphi}_i$, $i = 1,2$, defined in \cref{lem:Omega:definitions}, are continuous and \cref{lem:Omega:statement_1,lem:Omega:statement_2,lem:Omega:statement_3,lem:Omega:statement_4} hold. In accordance with \cref{lem:Omega:definitions}, denote $\Omega^\beta$ and $\kappa(\cdot)$. According to \cref{lem:eta}, define $\eta \in \mathrm{C}^1(\mathbb R \times \mathbb R^n, \mathbb R)$ and $\lambda_\eta > 0$ such that \cref{lem:eta:statement_1,lem:eta:statement_2,lem:eta:statement_3} are valid.

For every $\varepsilon > 0$, define the function
\begin{displaymath}
\begin{array}{c}
\chi_\varepsilon(t,x,\xi,y) = \varphi_1(t,x,\kappa_t(\cdot)) - \varphi_2(\xi,y,\kappa_\xi(\cdot)) - \big(|t - \xi|^2 + \|x - y\|^2\big) / \varepsilon\\[0.2cm]
- 2 (2 \nu - t - \xi) \theta - \eta(t,x) - \eta(\xi,y),\quad (t,x),(\xi,y) \in \Omega^\beta.
\end{array}
\end{displaymath}
Since $\tilde{\varphi}_i$, $i = 1,2$, and $\eta$ are continuous, the function $\chi_\varepsilon$ is continuous. Denote
\begin{equation}\label{lem:uniqueness:t_eps_x_eps_xi_eps_y_eps}
(t_\varepsilon, x_\varepsilon, \xi_\varepsilon, y_\varepsilon) \in \mathrm{argmax}\big\{\chi_\varepsilon(t,x,\xi,y) \,|\, (t,x),(\xi,y) \in \Omega^\beta\big\}.
\end{equation}
Then, from \cref{lem:Omega:statement_2,lem:eta:statement_1}, we derive
\begin{displaymath}
\begin{array}{c}
\big(|t_\varepsilon - \xi_\varepsilon|^2 + \|x_\varepsilon - y_\varepsilon\|^2\big) / \varepsilon
\leq \varphi_1(t_\varepsilon,x_\varepsilon,\kappa_{t_\varepsilon}(\cdot)) - \varphi_2(\xi_\varepsilon,y_\varepsilon,\kappa_{\xi_\varepsilon}(\cdot)) - \chi_\varepsilon(t_\varepsilon, x_\varepsilon, \xi_\varepsilon, y_\varepsilon) \\[0.2cm]
\leq 2 \alpha_\varphi - \chi_\varepsilon(\tau,z,\tau,z) \leq 4 \alpha_\varphi + 4 (\nu - \tau) \theta.
\end{array}
\end{displaymath}
Since the right-hand side of the inequality does not depend on $\varepsilon$, we obtain
\begin{equation}\label{lem:uniqueness:t0_x0}
t_\varepsilon, \xi_\varepsilon \to t_\circ,\quad x_\varepsilon, y_\varepsilon \to x_\circ,\quad \text{ as } \quad \varepsilon \to 0.
\end{equation}

Let us prove that  $\|x_\circ - z\| < (t_\circ - \tau) c_* + \beta$. For the sake of a contradiction, suppose that $\|x_\circ - z\| = (t_\circ - \tau) c_* + \beta$. Then, using \cref{lem:Omega:statement_2}, \cref{lem:eta:statement_1}, \cref{lem:uniqueness:t_eps_x_eps_xi_eps_y_eps}, \cref{lem:uniqueness:t0_x0}, and continuity of $\eta$, we derive
\begin{displaymath}
\begin{array}{c}
\varphi_1(\tau,z,w(\cdot)) - \varphi_2(\tau,z,w(\cdot)) - 4 (\nu - \tau) \theta = \chi_\varepsilon(\tau,z,\tau,z) \leq \limsup\limits_{\varepsilon \to +0} \chi_\varepsilon(t_\varepsilon,x_\varepsilon,\xi_\varepsilon,y_\varepsilon) \\[0.2cm]
\leq 2 \alpha_\varphi\! - 2 \eta(t_\circ,x_\circ)\! \leq\! - 2 \alpha_\varphi\! - 8 (\nu - \tau) \theta \leq \varphi_1(\tau,z,w(\cdot)) - \varphi_2(\tau,z,w(\cdot)) - 8 (\nu - \tau) \theta,
\end{array}
\end{displaymath}
that contradicts the inequality $\theta > 0$. Let us show that $t_\circ < \nu$. For the sake of a contradiction, suppose that $t_\circ = \nu$. Then, due to \cref{lem:Omega:statement_2}, \cref{lem:eta:statement_1}, \cref{lem:uniqueness:t_eps_x_eps_xi_eps_y_eps}, \cref{lem:uniqueness:t0_x0}, and a continuity of $\tilde{\varphi}_i$, $i = 1,2$ (see \cref{lem:Omega:definitions}), we obtain
\begin{displaymath}
\begin{array}{c}
\varphi_1(\tau,z,w(\cdot)) - \varphi_2(\tau,z,w(\cdot)) - 4 (\nu - \tau) \theta
= \chi_\varepsilon (\tau,z,\tau,z) \leq \limsup\limits_{\varepsilon \to +0} \chi_\varepsilon (t_\varepsilon, x_\varepsilon, \xi_\varepsilon, y_\varepsilon)\\[0.2cm]
\leq \varphi_1(\nu,x_\circ,\kappa_\nu(\cdot)) - \varphi_2(\nu,x_\circ,\kappa_\nu(\cdot)),
\end{array}
\end{displaymath}
that contradicts \cref{lem:Omega:statement_4}. Thus, there exist $\varepsilon_* > 0$ and $\delta > 0$ such that
\begin{equation}\label{lem:uniqueness:epsilon_s_delta}
O^+_\delta(t_\varepsilon,x_\varepsilon), O^+_\delta(\xi_\varepsilon,y_\varepsilon) \subset \Omega^\beta,\quad \varepsilon \in (0,\varepsilon_*),
\end{equation}
where $O^+_\delta(t_\varepsilon,x_\varepsilon)$ and $O^+_\delta(\xi_\varepsilon,y_\varepsilon)$ are defined according to \cref{def:viscosity_solution}.

Denote
\begin{equation}\label{lem:uniqueness:alpha_s}
B_s = \{s \in \mathbb R^n: \|s\| \leq \alpha_s\},\quad \alpha_s = 2 \lambda_\varphi + 3 \lambda_\eta.
\end{equation}
Due to a continuity of $\tilde{H}$ (see \cref{lem:Omega:definitions}) and \cref{lem:Omega:statement_1}, \cref{lem:uniqueness:t0_x0}, we can fix $\varepsilon \in (0, \varepsilon_*)$ such that
\begin{equation}\label{lem:uniqueness:H}
|H(t_\varepsilon,x_\varepsilon,\kappa_{t_\varepsilon}(\cdot),s) - H(\xi_\varepsilon,y_\varepsilon,\kappa_{\xi_\varepsilon}(\cdot),s')| \leq \theta + c_*\|s - s'\|,\quad s,s' \in B_s.
\end{equation}

Define the function
\begin{displaymath}
\psi_1(t,x) = - \chi_\varepsilon(t,x,\xi_\varepsilon,y_\varepsilon) + \varphi_1(t,x,\kappa_t(\cdot)) + (t - t_\varepsilon) \theta,\quad (t,x) \in \mathbb R \times \mathbb R^n.
\end{displaymath}
Then, according to the inclusion $\eta \in \mathrm{C}^1(\mathbb R \times \mathbb R^n, \mathbb R)$, we have $\psi_1 \in \mathrm{C}^1(\mathbb R \times \mathbb R^n, \mathbb R)$ and
\begin{displaymath}
\partial \psi_1(t,x) / \partial t = 2(t - \xi_\varepsilon) / \varepsilon + \partial \eta(t,x) / \partial t - \theta,\quad
\nabla_x \psi_1(t,x) = 2 (x - y_\varepsilon) / \varepsilon + \nabla_x \eta(t,x).
\end{displaymath}
Let $\kappa'(\cdot) \in \Lambda_0(t_\varepsilon,x_\varepsilon,\kappa_{t_\varepsilon}(\cdot))$. Then, taking into account \cref{lem:Omega:statement_3,lem:uniqueness:t_eps_x_eps_xi_eps_y_eps}, for every $(t,x) \in O^+_\delta(t_\varepsilon,x_\varepsilon)$, we derive
\begin{displaymath}
\begin{array}{c}
\varphi_1(t,x,\kappa'_t(\cdot)) - \psi_1(t,x) \leq
\varphi_1(t,x,\kappa_t(\cdot)) - \psi_1(t,x) + (t - t_\varepsilon) \theta = \\[0.2cm]
\chi_\varepsilon(t,x,\xi_\varepsilon,y_\varepsilon) \leq \chi_\varepsilon(t_\varepsilon,x_\varepsilon,\xi_\varepsilon,y_\varepsilon) = \varphi_1(t_\varepsilon,x_\varepsilon,\kappa_{t_\varepsilon}(\cdot)) - \psi_1(t_\varepsilon,x_\varepsilon).
\end{array}
\end{displaymath}
Thus, since $\varphi_1$ is a viscosity solution of problem \cref{Hamilton-Jacobi_equation,terminal_condition} (see \cref{def:viscosity_solution}), we obtain
\begin{equation}\label{lem:uniqueness:sub_solution}
2(t_\varepsilon - \xi_\varepsilon) / \varepsilon + \partial \eta(t_\varepsilon,x_\varepsilon) / \partial t - \theta
+ H\big(t_\varepsilon,x_\varepsilon,\kappa_{t_\varepsilon}(\cdot), 2 (x_\varepsilon - y_\varepsilon) / \varepsilon + \nabla_x \eta(t_\varepsilon,x_\varepsilon)\big) \geq 0.
\end{equation}
In the similar way, using the function
\begin{displaymath}
\psi_2(\xi,y) = \chi_\varepsilon(t_\varepsilon, x_\varepsilon, \xi, y) + \varphi_2(\xi, y, \kappa_\xi(\cdot)) - (\xi - \xi_\varepsilon) \theta,\quad (\xi,y) \in \mathbb R \times \mathbb R^n,
\end{displaymath}
we obtain
\begin{equation}\label{lem:uniqueness:super_solution}
2(t_\varepsilon - \xi_\varepsilon) / \varepsilon - \partial \eta(\xi_\varepsilon,y_\varepsilon) / \partial \xi + \theta
+ H\big(\xi_\varepsilon,y_\varepsilon,\kappa_{\xi_\varepsilon}(\cdot), 2 (x_\varepsilon - y_\varepsilon) / \varepsilon - \nabla_y \eta(\xi_\varepsilon,y_\varepsilon)\big) \leq 0.
\end{equation}

Let us consider the case when $t_\varepsilon \geq \xi_\varepsilon$. Then, according to \cref{lem:uniqueness:t_eps_x_eps_xi_eps_y_eps}, we have $\chi_\varepsilon(t_\varepsilon,x_\varepsilon,\xi_\varepsilon,y_\varepsilon) \geq \chi_\varepsilon(t_\varepsilon,y_\varepsilon,\xi_\varepsilon,y_\varepsilon)$. Hence, using \cref{lem:Omega:statement_2,lem:eta:statement_3}, we derive
\begin{displaymath}
\begin{array}{c}
0 \leq \varphi_1(t_\varepsilon,x_\varepsilon,\kappa_{t_\varepsilon}(\cdot)) - \varphi_1(t_\varepsilon,y_\varepsilon,\kappa_{t_\varepsilon}(\cdot))
- \|x_\varepsilon - y_\varepsilon\|^2 / \varepsilon - \eta(t_\varepsilon,x_\varepsilon) + \eta(t_\varepsilon,y_\varepsilon) \\[0.2cm]
 \leq (\lambda_\varphi + \lambda_\eta) \|x_\varepsilon - y_\varepsilon\| - \|x_\varepsilon - y_\varepsilon\|^2 / \varepsilon,
\end{array}
\end{displaymath}
wherefrom, we obtain the estimate $\|x_\varepsilon - y_\varepsilon\| \leq (\lambda_\varphi + \lambda_\eta) \varepsilon$. This estimate can be also obtained for the case when $t_\varepsilon < \xi_\varepsilon$ using  $\chi_\varepsilon(t_\varepsilon,x_\varepsilon,\xi_\varepsilon,y_\varepsilon) \geq \chi_\varepsilon(t_\varepsilon,x_\varepsilon,\xi_\varepsilon,x_\varepsilon)$. Based on this estimate together with \cref{lem:eta:statement_3}, we get
\begin{equation}\label{lem:uniqueness:c_s}
\max\big\{\big\|2 (x_\varepsilon - y_\varepsilon) / \varepsilon + \nabla_x \eta(t_\varepsilon,x_\varepsilon)\big\|,\ \big\|2 (x_\varepsilon - y_\varepsilon) / \varepsilon - \nabla_y \eta(\xi_\varepsilon,y_\varepsilon)\big\|\big\} \leq \alpha_s.
\end{equation}
where $\alpha_s$ is defined by \cref{lem:uniqueness:alpha_s}. Then, using \cref{lem:eta:statement_2,lem:uniqueness:H,lem:uniqueness:sub_solution,lem:uniqueness:super_solution}, we conclude
\begin{displaymath}
\begin{array}{rcl}
0 \!\!\!&\leq&\!\!\!\! \partial \eta(t_\varepsilon,x_\varepsilon) / \partial t - \theta
+ H\big(t_\varepsilon,x_\varepsilon,\kappa_{t_\varepsilon}(\cdot), 2 (x_\varepsilon - y_\varepsilon) / \varepsilon + \nabla_x \eta(t_\varepsilon,x_\varepsilon)\big)\\[0.2cm]
\!\!\!&+&\!\!\!\! \partial \eta(\xi_\varepsilon,y_\varepsilon) / \partial \xi - \theta
- H\big(\xi_\varepsilon,y_\varepsilon,\kappa_{\xi_\varepsilon}(\cdot), 2 (x_\varepsilon - y_\varepsilon) / \varepsilon - \nabla_y \eta(\xi_\varepsilon,y_\varepsilon)\big)\\[0.2cm]
\!\!\!&\leq&\!\!\!\! \partial \eta(t_\varepsilon,x_\varepsilon) / \partial t + \partial \eta(\xi_\varepsilon,y_\varepsilon) / \partial \xi + c_*\big(\|\nabla_x \eta(t_\varepsilon,x_\varepsilon)\| + \|\nabla_y \eta(\xi_\varepsilon,y_\varepsilon)\|\big) - \theta \leq - \theta,
\end{array}
\end{displaymath}
that contradicts the inequality $\theta > 0$.
\end{proof}

\begin{remark}
From the prove of the lemma above, it follows that a viscosity solution of problem \cref{Hamilton-Jacobi_equation,terminal_condition} is unique for any Hamiltonian $H$ satisfying conditions of \cref{lem:Omega}.
\end{remark}

%%%%%%%%%%%%%%%%%%%%%%%%%%%%%%%%%%%%%%%%%%%%%%%%%%%%%%%%%%%%%%%%%%%%%%%%%%%%%%%%%%%%%%%%
%%%%%%%%%%%%%%%%%%%%%%%%%%%%%%%%%%%%%%%%%%%%%%%%%%%%%%%%%%%%%%%%%%%%%%%%%%%%%%%%%%%%%%%%
%%%%%%%%%%%%%%%%%%%%%%%%%%%%%%%%%%%%%%%%%%%%%%%%%%%%%%%%%%%%%%%%%%%%%%%%%%%%%%%%%%%%%%%%
%%%%%%%%%%%%%%%%%%%%%%%%%%%%%%%%%%%%%%%%%%%%%%%%%%%%%%%%%%%%%%%%%%%%%%%%%%%%%%%%%%%%%%%%
%%%%%%%%%%%%%%%%%%%%%%%%%%%%%%%%%%%%%%%%%%%%%%%%%%%%%%%%%%%%%%%%%%%%%%%%%%%%%%%%%%%%%%%%
%%%%%%%%%%%%%%%%%%%%%%%%%%%%%%%%%%%%%%%%%%%%%%%%%%%%%%%%%%%%%%%%%%%%%%%%%%%%%%%%%%%%%%%%

\section{Infinitesimal criterion of viscosity solutions} In the section, we give lemmas which imply \cref{prop:subdifferential_inequalities}. These lemmas will be proved similarly to \cite[Lemma~1.1]{Crandall_Evans_Lions_1984}.

\begin{lemma}\label{lem:sub_viscosity_solution_and_subdifferential_viscosity_solution}
Let $\varphi \in \Phi$. Then property \cref{sub_viscosity_solution} is equivalent to inequality \cref{subdifferential_viscosity_solution}.
\end{lemma}
\begin{proof}
Let us show that \cref{sub_viscosity_solution} implies  \cref{subdifferential_viscosity_solution}. Let $(\tau,z,w(\cdot)) \in \mathbb G$, $\tau < \vartheta$ and $(p_0,p) \in D^-\varphi(\tau,z,w(\cdot))$. Define the function $\tilde{\varphi}$ by \cref{tilde_phi} and the function $\overline{\varphi}$ as follows
\begin{displaymath}
\begin{array}{lll}
\overline{\varphi}(t,x) \!\!\!\!\!& = \min\{0,\tilde{\varphi}(t,x) - \tilde{\varphi}(\tau,z) - (t - \tau) p_0 - \langle x - z, p \rangle\}, &(t, x) \in [\tau,\vartheta] \times \mathbb R^n,\\[0.2cm]
\overline{\varphi}(t,x) \!\!\!\!\!& = \overline{\varphi}(\vartheta,x), &(t, x) \in (\vartheta,+\infty) \times \mathbb R^n,\\[0.2cm]
\overline{\varphi}(t,x) \!\!\!\!\!& = \overline{\varphi}(\tau - (t - \tau),x), &(t, x) \in (-\infty,\tau) \times \mathbb R^n.
\end{array}
\end{displaymath}
Taking into account \cref{lem:tilde_phi}, we can show that $\overline{\varphi}$ is continuous on $\mathbb R \times \mathbb R^n$, differentiable at $(\tau,z)$, and $\overline{\varphi}(\tau,z) = 0$, $\partial\, \overline{\varphi}(\tau,z) / \partial \tau = 0$, $\nabla_z \overline{\varphi}(\tau,z) = 0$. By \cite[Lemma~I.4]{Crandall_Lions_1983}, there exists a function $\overline{\psi} \in \mathrm{C}^1(\mathbb R \times \mathbb R^n, \mathbb R)$ such that $\overline{\psi}(\tau,z) = 0$, $\partial\, \overline{\psi}(\tau,z) / \partial \tau = 0$, $\nabla_z \overline{\psi}(\tau,z) = 0$, and
\begin{displaymath}
\overline{\varphi}(t,x) - \overline{\psi}(t,x) \geq \overline{\varphi}(\tau,z) - \overline{\psi}(\tau,z),\quad (t,x) \in O^+_\delta(\tau,z),
\end{displaymath}
for some $\delta > 0$. Then, defining $\psi(t,x) = \overline{\psi}(t,x) + (t - \tau) p_0 + \langle x - z, p \rangle$, we have $\partial\psi(\tau,z) / \partial \tau = p_0$, $\nabla_z \psi(\tau,z) = p$, and hence, from \cref{sub_viscosity_solution} we obtain \cref{subdifferential_viscosity_solution}.

Let us show that \cref{subdifferential_viscosity_solution}  implies \cref{sub_viscosity_solution}. Let $(\tau,z,w(\cdot)) \in \mathbb G$, $\tau < \vartheta$, $\psi \in \mathrm{C}^1(\mathbb R \times \mathbb R^n, \mathbb R)$, and $\delta > 0$ be such that
\begin{displaymath}
\varphi(\tau,z,w(\cdot)) - \psi(\tau,z) \leq \varphi(t,x,\kappa_t(\cdot)) - \psi(t,x),\quad (t,x) \in O^+_\delta(\tau,z).
\end{displaymath}
Hence, according to definition of $D^-\varphi(\tau,z,w(\cdot))$, we derive $(\partial \psi(\tau,z) / \partial \tau, \nabla_z \psi(\tau,z)) \in D^-\varphi(\tau,z,w(\cdot))$. Applying this inclusion to \cref{subdifferential_viscosity_solution}, we obtain \cref{sub_viscosity_solution}.
\end{proof}

\begin{lemma}\label{lem:super_viscosity_solution_and_superdifferential_viscosity_solution}
Let $\varphi \in \Phi$. Then property \cref{super_viscosity_solution} is equivalent to inequality \cref{superdifferential_viscosity_solution}.
\end{lemma}
\begin{proof}
The lemma can be proved similarly to \cref{lem:sub_viscosity_solution_and_subdifferential_viscosity_solution}
\end{proof}

%\section*{Acknowledgments}

\bibliographystyle{siamplain}
\bibliography{references}

\end{document}